\documentclass[12 pt]{amsart}
\usepackage{amssymb}
\usepackage{graphicx}
\usepackage[mathscr]{eucal}
\usepackage[usenames]{color}
\definecolor{Purple}{rgb}{.7,0.08,0.6} 
 
\usepackage{enumerate}
\newenvironment{enum}{  
\begin{enumerate}[\upshape(\arabic{section}.\arabic{equation}a)] }  
{  \end{enumerate}   }
\newcommand{\itemref}[2] {{\upshape(\ref{#1}\ref{#2})}}

\theoremstyle{plain}
\newtheorem{Thm}{Theorem}
\newtheorem{Cor}[Thm]{Corollary}
\newtheorem{Prop}[Thm]{Proposition}
\newtheorem{Lem}[Thm]{Lemma}

\newtheorem{Def}[Thm]{Definition}

\newtheorem{Remark}[Thm]{Remark}
\theoremstyle{definition}

\newtheorem{Ex}{Example}

\newcommand{\PP}{\mathscr P}
\newcommand{\RR}{\mathscr R}
\newcommand{\Lop}{\mathbb L}
\newcommand{\Hop}{\mathbb H}
\newcommand{\Aop}{\mathbb A}
\newcommand{\Aopnm}{\mathbb A_{n,m}}
\newcommand{\levi}{\mathscr L}
\newcommand{\rr}{R}

\renewcommand{\Re}{\operatorname{Re}}
\renewcommand{\bar}{\overline}
\renewcommand{\tilde}{\widetilde}
\newcommand{\C}{\mathbb C}
\newcommand{\R}{\mathbb R}
\newcommand{\Z}{\mathbb Z}
\newcommand{\N}{\mathbb N}
\newcommand{\dee}{\partial}
\newcommand{\deebar}{\overline\partial}
\newcommand{\st}{\,:\,}
\DeclareMathOperator*{\Span}{Span}
\newcommand{\lam}{\lambda}
\newcommand{\eps}{\epsilon}
\newcommand{\eqdef}{\overset{ \text{def} }{=}}
\newcommand{\bndry}{b}
\newcommand{\w}{\wedge}
\newcommand{\inv}{^{-1}}
\numberwithin{equation}{section}

\begin{document}
\title[The spectrum of the Leray transform]
{The spectrum of the Leray transform for convex Reinhardt domains in $\C^2$}
\author[David E. Barrett and Loredana Lanzani]{David E. Barrett and Loredana Lanzani$^*$}
\thanks{$^*$
Supported in part by the AWM and the NSF (grant No.
DMS-0700815.)}
\address{Dept. of Mathematics\\University of Michigan
\\Ann Arbor, MI  48109-1043  USA }
\address{Dept. of Mathematics \\        
University of Arkansas 
Fayetteville, AR 72701
  USA }
\date{\today}
\email{barrett@umich.edu,\; lanzani@uark.edu}
\begin{abstract}  The Leray transform and related boundary operators are studied for a class of convex Reinhardt domains in $\mathbb C^2$. Our class is self-dual; it contains some domains with less than $C^2$-smooth boundary and also some domains with smooth boundary and 
degenerate Levi form.
 $L^2$-regularity is proved, and essential spectra are computed with respect to a family of boundary measures which includes surface measure. 
  A duality principle is established providing explicit unitary
 equivalence between operators on domains in our class and operators on the corresponding polar domains.
 Many of these results are new even for the classical case of smoothly bounded strongly convex Reinhardt domains.
  
\end{abstract}


\maketitle

\section{Introduction}\label{S:int}

The Leray transform $\Lop$ is a higher-dimensional analog of the classical Cauchy transform for planar domains.  It belongs to a family of operators, the Cauchy-Fantappi\'e transforms, projecting functions on the boundary onto the space of holomorphic boundary values.  These operators play an essential role in higher-dimensional function theory, just as the original Cauchy transform does in the one-dimensional setting.  (See for instance Kerzman and Stein \cite{KS1} and the monographs \cite{HeLe}, \cite{Kra} and\cite{Ran}.)  

Though the Cauchy-Fantappi\'e construction is not canonical in general, the Leray transform is distinguished by the simple explicit construction of the corresponding kernel function and by the presence of a good transformation law under linear fractional transformations (\cite{Bol2}, Thm. 3).  The construction of the Leray transform requires that the domain under study satisfy the geometric condition of ``$\C$-linear convexity." 

In this paper we provide rather detailed information about the Leray transform on certain convex Reinhardt domains in
 $\mathbb C^2$.
In particular, we learn that
\begin{itemize}
\item[(A)] $\Lop$ is $L^2$-bounded on some, but not all, smoothly bounded weakly convex domains;
\item[(B)] $\Lop$ is $L^2$-bounded on some, but not all, strongly convex domains whose boundaries are less than $C^2$-smooth;
\item[(C)] it is important to give thought to the choice of boundary measure -- in particular,  measures involving (suitably-chosen) powers of the Levi form work as well as (or better than)
surface measure;
\item[(D)]  there is a duality rule relating the qualitative and quantitative behavior of $\Lop$ on a domain $D$ to the corresponding behavior on the polar domain $D^*$ (defined in \eqref{E:polar-def}). This provides a surprising linkage between the previous topics (A) and (B).
\end{itemize}

The  {\em Reinhardt} designation means that $D$ is invariant under all rotations of the form 
\begin{equation}\label{E:rot}
(z_{1},z_{2})\mapsto(e^{i\theta_{1}}z_{1},e^{i\theta_{2}}z_{2}).
\end{equation}
Reinhardt domains occur naturally in  various contexts in several complex variables (for instance, the domains of convergence of power series of holomorphic functions are Reinhardt domains) and are often a source of meaningful examples which serve as models for more general theories.
One class of domains singled out in our work is the class $\tilde\RR$ consisting of bounded convex complete $C^1$-smooth Reinhardt domains in $\C^{2}$  that are $C^2$-smooth and strongly convex away from the axes $\{\zeta_{1}\zeta_{2}=0\}$.
  (See Proposition \ref{P:D-phi} for an alternate description of $\tilde\RR$.)

The class $\tilde\RR$ contains the subclass $\PP$ consisting of weighted
$L^p$-balls; 
that is,
\begin{equation}\label{E:PPdom}
\PP=\{D_{p,a_{1},a_{2}}\st a_1>0, a_2>0, 1<p<\infty\},
\end{equation}
where we have let
\begin{equation}\label{E:Pball}
D_{p,a_{1},a_{2}}\,=\,\{(z_1,z_2)\in \C^2\st a_1|z_1|^p + a_2|z_2|^p<1\}.\end{equation}

Finally, we let 
$\RR$ denote the class of domains in $\tilde\RR$ that are well-modeled by a domain $D_{p_j,a_{1,j},a_{2,j}}\in \PP$ near boundary points on each of the axes $\zeta_j=0$, $j=1, 2$.  (See Definition \ref{D:rr} for the formal description.)

We have $\PP \subsetneq \RR \subsetneq \tilde\RR.$

The smoothness of a domain in $\RR$ is determined by the size of the exponents $p_1,p_2$. On the one hand, if $1<p_{1}<2$,  an $\RR$-domain will be strongly geometrically convex (in the sense of \cite{Pol}) and  $C^{1, p_1-1}$-smooth near  $\{\zeta_{1}=0\}$; on the other hand,  for
$p_{1}\ge2$ the domain will be at
least $C^{2}$-smooth near $\{\zeta_{1}=0\}$, but strong geometric convexity and strong Levi pseudoconvexity will fail   if $p_{1}>2$.
The size of $p_{2}$ similarly determines the qualitative behavior of the domain near $\{\zeta_{2}=0\}$.

We will show in Proposition \ref{P:osc} below that for $D\in\tilde\RR$ and $\zeta\in\bndry D\setminus\{\zeta_1\zeta_2=0\}$ there is a unique $D_{p(\zeta),a_1(\zeta),a_2(\zeta)}\in\PP$ osculating $D$  at $\zeta$ in the sense that all data up through second order will match there.
 If $D\in \RR$, then setting $p(\zeta)=p_{1}$ when $\zeta_{1}=0$ and $p(\zeta)=p_{2}$ when $\zeta_{2}=0$ we get a continuous function $p(\zeta)$ defined on all of $\bndry D$ (see Proposition \ref{P:RRdom}).

For a $C^2$-smooth convex domain $D$ in $\C^2$ the 
{\em Leray integral} $\Lop=\Lop_{D}$
  is defined  by letting
\begin{equation}\label{E:Lop}
\Lop f(w) = \int\limits_{\zeta\in\bndry D}
\!\!\!\!f(\zeta) \,L(\zeta, w)
\end{equation}
for $w\in D$, where  
 \begin{equation}\label{E:Lker}
L(\zeta, w)=\frac{1}{(2\pi i)^2}\,\frac{j^*(\dee\rho\w\deebar\dee\rho)(\zeta)}{[\dee\rho(\zeta)\bullet(\zeta-w)]^{2}}
\end{equation}
is the {\em Leray kernel}  defined for $\zeta \in \bndry D, 
w \in D$;
here $\rho$ is a defining function for $\bndry D$, 
$j^*$ denotes the pullback of the inclusion $j: \bndry D \to \C^2$ acting on three-forms, and $\dee\rho(\zeta)\bullet(\zeta-w)$ denotes the action of the linear functional $\dee\rho(\zeta)$ on the vector $\zeta-w$, namely
\begin{equation}\label{E:gen}
\dee\rho(\zeta)\bullet(\zeta-w) =
\frac{\dee \rho}{\dee \zeta_{1}}(\zeta) (\zeta_{1}-w_{1})+\frac{\dee \rho}{\dee \zeta_{2}}(\zeta) (\zeta_{2}-w_{2}).
\end{equation} 
It follows from the convexity of $D$ that $\partial\rho (\zeta)$ is a so-called ``generating form'' for $D$; if $\bndry D$ contains no line segments we have in particular that the expression in \eqref{E:gen} is non-zero for each $\zeta\in bD$ and for each $w\in\overline D\setminus \{\zeta\}$  (see \cite{Ran}, \S IV.3.1 and \S IV.3.2).

The kernel $L(\zeta, w)$ is independent of the choice of defining function $\rho$ (see \cite{Ran},  \S IV.3.2, also \cite{Ler}, \cite{Nor},\cite{Aiz}).  

The function $\Lop f$ will be holomorphic in $D$ when the integral \eqref{E:Lop} converges, and $\Lop$ reproduces a holomorphic function from its boundary values.  

\medskip

We should mention that the Leray integral is defined more generally for $\C$-linearly convex domains, that is, for domains whose complement is a union of complex hyperplanes.  (These are also known as 
``lineally convex" domains.) But  $\C$-linearly convex complete Reinhardt domains are automatically convex (see Example 2.2.4 in \cite{APS}) so in the current work we focus only on convex Reinhardt domains.

\medskip

When $D$ satisfies additional hypotheses (e.g. strong convexity) then $\Lop$ extends to a singular integral operator on the boundary, also denoted by $\Lop$ (see \cite{KS1}, page 207, and \cite{LS}).

\medskip

For domains $D\in\tilde\RR$ the theory outlined above does not apply directly, but we will show in particular that the reproducing property for holomorphic functions is still valid (see Corollary \ref{C:poly-rep} and Proposition \ref{P:Hardy}). 

In order to consider bounds and adjoints for $\Lop$ we will need to introduce measures on $\bndry D$; specifically, we will consider measures $\mu$ that are invariant under the rotations \eqref{E:rot} and are absolutely continuous with respect to surface measure.
We will take particular interest in boundary measures
that are continuous positive multiples
of 
$|\levi(\zeta)|^{1-q}\,d\sigma(\zeta)$, where $q$ is a fixed  real exponent,  $d\sigma$ is surface measure and $|\levi|$ is the Euclidean norm
\begin{equation}\label{E:levinorm}
|\levi|=
-\dfrac{j^*(\dee\rho\w\deebar\dee\rho)}{|\nabla\rho|^2\,d\sigma}
\end{equation}
of the Levi-form.
(Here we interpret the three-form $j^*(\dee\rho\w\deebar\dee\rho)$ as a measure on $\bndry D$.)  We will say that such a measure has {\em order $q$} (see Definition \ref{D:ord-q} below).

We are ready now to state our main results.

\begin{Thm}\label{T:rress}
Suppose $D\in\RR$ 
and $\mu$ is a rotation-invariant boundary measure  of order $q$ with $q$ satisfying the condition 
\begin{equation}\label{E:rc0}
   |q|<
    \min\limits_{j=1, 2}\left|\frac{p_j}{p_j-2}\right|=
   \min\limits_{j=1, 2}{\left|\frac{1}{p_j} - \frac{1}{p^*_j}\right|}^{-1}.
   \end{equation}
(Here  $p_1$ and $p_2$ are as in the description of $\RR$ above,  and $p_j^*$ denotes the conjugate exponent to $p_j$ -- thus $1/p_j + 1/p_j^* =1$.)

Then  the Leray transform $\Lop$
is bounded on $L^{2}(\bndry D,\mu)$.

Moreover, the operator $\Lop_\mu^*\Lop$ admits an orthogonal basis of eigenfunctions, and  the essential spectrum of $\Lop_\mu^*\Lop$ is equal to
\begin{equation}\label{E:spectrum}
\{0\}\cup\left\{\frac{\sqrt{p(\zeta)\,p^{*}(\zeta)}}{2}\st \zeta\in\bndry D\right\}
\cup \left\{\lambda_{p_{j},q,n}\st j=1,2, \, n\ge0 \right\};
\end{equation}
here,  $\Lop_\mu^*$ is the adjoint of $ \Lop$ in $L^{2}(\bndry D,\mu)$, $p(\zeta)$
 is the function discussed above (see Proposition \ref{P:osc} and Proposition \ref{P:RRdom}),
 $p^*(\zeta)$ denotes the conjugate exponent to $p(\zeta)$ (thus 
$1/p(\zeta) + 1/p^*(\zeta) =1$), and 
\begin{equation}\label{E:gamma}
\lambda_{p,q,n}
=
\frac{\Gamma\left(\frac{2n}{p}+1+q\left(\frac{1}{p}-\frac{1}{p^{*}}\right)\right)
\Gamma\left(\frac{2n}{p^{*} }+1+q\left(\frac{1}{p^{*}}-\frac{1}{p}\right)\right)
}
{\Gamma^{2}(n+1)
\left(\frac{2}{p}\right)^{\frac{2n}{p }+1+q\left(\frac{1}{p}-\frac{1}{p^{*}}\right)}
\left(\frac{2}{p^{*}}\right)^{\frac{2n}{p^{*}}+1+q\left(\frac{1}{p^{*}}-\frac{1}{p}\right)}
}.
\end{equation}
\end{Thm}

(For the definition and basic properties of the essential norm and the essential spectrum, see Propositions \ref{R:ess-spec} and
 \ref{R:ess-norm} and adjacent material). 
Note that the interval $|q|\leq 1$ is always included in \eqref{E:rc0}.  In Corollary \ref{P:smdom} below we will show that if $D \in \RR$ is a smooth domain then  $p_1=p_2=2$ so that \eqref{E:rc0} holds for \emph{all} values $q\in \mathbb R$.  On the other hand,  if at least one of the $p_j$ is different from 2 then \eqref{E:rc0} defines a proper subinterval of the real line. 

 Theorem \ref{T:rress} may be compared with previous work by 
Bonami and Lohou\'e \cite{BL} and Hansson \cite{Han} (which we specialize here to complex dimension $n=2$), as follows.  Given $1<p_j<+\infty$ set
$$
 D=\{(\zeta_1,\zeta_2)\in\mathbb C^2 \st |\zeta_1|^{p_{_1}}+|\zeta_2|^{p_{_2}}<1\}.
 $$
Note that $D$ belongs to the class  $\RR$. 
Bonami and Lohou\'e study Cauchy-Fantappi\'e
transforms and related operators for $D$ as above 
 when $p_j>2$, $j=1,2$ and $\mu$ is a  measure of order $q=1$.
Hansson proves that for $D$ as above the operator
 $\Lop$ is bounded on $L^2(\bndry D,\mu)$
 when $p_j>2$ are positive integers and $\mu$ is a measure of order $q=0$. In either case $\bndry D$ is $C^{k}$-smooth ($k\ge 2$) and weakly pseudoconvex (its Levi form is singular at boundary points that lie along the axes $\{\zeta_1\zeta_2=0\}$). When $D\in\RR$ is as above but $p_j<2$ it follows that $D$ is strongly convex but non-smooth and  the construction of the Cauchy-Fantappi\'e kernels investigated by Bonami and Lohou\' e becomes problematic (see comments below after Corollary \ref {C:smksess}), whereas the Leray transform $\Lop$ is still well defined and by Theorem \ref{T:rress} it is bounded in $L^2(\bndry D, \mu)$ for all measures $\mu$ of order $q$ with $q$ ranging in the interval \eqref{E:rc0}.
 In fact, more is true:
in \S \ref{S:rr} we present a duality result providing an explicit unitary equivalence of the Leray transform for a domain $D\in\tilde\RR$ (resp. $D\in \RR$) and the Leray transform for its polar domain $D^*\in\tilde\RR$ (resp. $D^*\in\RR$). On the one hand, we see that
the polar of a smooth, weakly pseudoconvex domain may be non-smooth and strongly
convex;
for example, the polar of the domain $D\in\RR$ given above with $p_j>2$ is
$$
 D^*=\{(\zeta_1,\zeta_2)\st |\zeta_1|^{p^*_1}+|\zeta_2|^{p^*_2}<1\}\in\RR
$$
where $p^*_j <2$ is the conjugate exponent of $p_j$;
see Theorem \ref{T:duality} for the precise statement in the general case. On the other hand,
combining this duality with \eqref{E:spectrum} and \eqref{E:gamma} in Theorem \ref{T:rress} 
we see that, modulo a 
switch of measure (from $\mu$ of order $q$ to $\tilde\mu$ of order $-q$),
from the point of view of the spectral theory of the Leray transform any domain $D$ in the class $\RR$  is 
 qualitatively and quantitatively indistinguishable from its polar domain $D^*$.

 Lanzani and Stein show in \cite{LS} that $\Lop$ is $L^2$-bounded with respect to surface measure when $D$ is a bounded strongly ($\C$-linearly) convex domain in $\C^n$ with $C^{1,1}$-smooth boundary.  The examples discussed above show that neither strong convexity nor  $C^{1,1}$-smoothness of the boundary is a necessary condition for $L^2$-boundedness of $\Lop$.  On the other hand, in \S \ref{S:ce} we present examples showing that if we  
try to settle for weak convexity or  $C^{1,\alpha}$-smoothness of the boundary with no further conditions then $\Lop$ may fail to be $L^2$-bounded with respect to any reasonable boundary measure.

\medskip

Following Kerzman and Stein (\cite{KS1}, \cite{KS2}) we will use the notation $\Aop_\mu$
 for the anti-self-adjoint operator $\Lop_\mu^*-\Lop$.

\begin{Thm}\label{T:rress-ks} In the setting of Theorem \ref{T:rress},  the operator $\Aop_\mu$ admits an orthogonal basis of eigenfunctions, and  the essential spectrum  of $\Aop_\mu$  is equal to
 \begin{multline*}
\{0\}\cup\left\{\pm i 
\sqrt{\frac{\sqrt{p(\zeta)p^{*}(\zeta)}}{2}-1}\st \zeta\in\bndry D\right\}\\
\cup \left\{\pm i\sqrt{\lambda_{p_{j},q,n}-1}\st j=1,2, \, n\ge0 \right\}.
\end{multline*}
\end{Thm}

As mentioned above,  Corollary \ref{P:smdom} below will show that if $D \in \RR$ is a smooth domain, then 
$p_1=p_2=2$ so that for all $q\in \R$ we have  $\lambda_{p_{j},q,n}  =1$, $j=1,2$; thus the choice of $q$ is no longer relevant in the description of our class of measures and we obtain the following results.

\begin{Thm}\label{T:smless}
Suppose  $D\subset \C^2$ is a $C^2$-smooth, strongly convex Reinhardt domain and let $\mu$ be any rotation-invariant continuous positive multiple of surface measure.

Then  $\Lop$ is bounded on $L^2(\bndry D, \mu)$.

Moreover, the operator $\Lop_\mu^*\Lop$ admits an orthogonal basis of eigenfunctions, and  the essential spectrum of $\Lop_\mu^*\Lop$ is equal to 
\begin{equation*}
\{0\}\cup\left\{\frac{\sqrt{p(\zeta)\,p^{*}(\zeta)}}{2}\st \zeta\in\bndry D\right\}\ 
\end{equation*}
or, equivalently, 
\begin{equation*}
\{0\}\cup\left\{\frac{p(\zeta)}{2\sqrt{p(\zeta)-1}}\st \zeta\in\bndry D\right\} .
\end{equation*}
The essential norm of $\Lop$ is 
\begin{equation*}
\max\left\{\sqrt[4]{\frac{p(\zeta)\,p^{*}(\zeta)}{4}}\st \zeta\in\bndry D\right\}.
\end{equation*}
\end{Thm}

\begin{Thm}\label{T:smksess}
In the setting of Theorem \ref{T:smless}, the operator $\Aop_\mu$ admits an orthogonal basis of eigenfunctions, and  the essential spectrum  of $\Aop_\mu$  is equal to 
\begin{equation*}
\{0\}\cup\left\{\pm i \sqrt{\frac{\sqrt{p(\zeta)\,p^{*}(\zeta)}}{2}-1}\st \zeta\in\bndry D\right\}\  
\end{equation*}
or, equivalently,
\begin{equation*}
\{0\}\cup\left\{\pm i \sqrt{\frac{ p(\zeta)}{2\sqrt{ p(\zeta)-1}}-1}\st \zeta\in\bndry D\right\}.
\end{equation*}
The essential norm of $\Aop_\mu$ is
\begin{equation*}
\max\left\{\sqrt{\frac{\sqrt{p(\zeta)\,p^{*}(\zeta)}}{2}-1}\st \zeta\in\bndry D\right\}.
\end{equation*} 

\end{Thm}

Combining Theorem \ref{T:smksess} with Proposition \ref{P:pcon} below we obtain the following.

\medskip
\begin{Cor}\label{C:smksess}
In the setting of Theorem \ref{T:smksess}, 
 the operator $\Aop_\mu$ will be compact in $L^2(\bndry D, \mu)$ if and only if $D$ is a domain of the form \[\{(z_1,z_2)\st a_1|z_1|^2 + a_2|z_2|^2<1\}.\]
\end{Cor}

(See the end of \S \ref{S:Lt} for related results.)

\medskip

The results outlined above should be contrasted with work of
Kerzman and Stein \cite{KS1} on a related operator $\Hop$ of Cauchy-Fantappi\`e type due to Henkin \cite{Hen} and Ram\'irez \cite{Ram}. This operator is based on the (quadratic) Levi polynomial rather than the linear functions of $w$ appearing in \eqref{E:gen};  it may be defined on any strongly pseudoconvex domain with $C^3$-smooth boundary.  Kerzman and Stein show that the operator $\Hop$ is a compact perturbation of the Szeg\H o projection defined with respect to surface measure $\sigma$  (see the end of \S \ref{S:Lt}); it follows that $\Hop_\sigma^* \Hop$ has essential spectrum $\{0,1\}$, and $\Hop_\sigma^*-\Hop$ has essential spectrum $\{0\}$. Thus $\Hop$ provides more direct access to the Szeg\H o projection, while  $\Lop$ has a more informative spectral theory.

\medskip

The plan of the paper proceeds as follows.  

In \S \ref{S:geo} we provide more details about the classes of domains under study, relating properties of a Reinhardt domain $D$ to the geometry of the curve $\gamma_+=\bndry D\cap\R^2_+$.  We prove the osculation results mentioned above; the corresponding exponent $p$ defines a continuous map from $\gamma_+$  to the interval $(1,\infty)$.  We also introduce a special parameter $s$ on $\gamma_+$ which plays an important role throughout the rest of the paper, and we characterize the classes  $\tilde \RR$, $\RR$ and $\PP$ in terms of $p$ as a function of $s$.

In \S \ref{S:Lt}
we 
 present the basic theory of the Leray transform for domains in the class $\tilde \RR$, confirming in particular that the reproducing property for holomorphic functions still holds even when the domains are less than $C^2$-smooth.  We introduce a special class of measures on $\bndry D$, the {\em admissible} measures; in essence, a rotation-invariant measure $\mu$ on $\bndry D$ is admissible if and only if $\mu$ is finite and $\Lop$ maps $L^2(\bndry D,\mu)$ to holomorphic functions on $D$. We also discuss norms of the  Fourier pieces of $\Lop$ (and of $\Lop^*_\mu\Lop$ and $\Aop_\mu$) and explain their relation to properties of the overall operators.

\S \ref{S:bdm} contains more information about boundary measures and geometry, confirming in particular that for $D\in\RR$ a measure of order $q$ is admissible if and only if condition \eqref{E:rc0} holds.

In \S \ref{S:as} we perform some asymptotic analysis of the norms of the Fourier pieces and use these results to prove Theorems \ref {T:rress} and \ref {T:rress-ks}.

\S \ref{S:ce} contains examples of domains for which the $L^2$-boundedness of the Leray transform fails (with respect to any admissible measure, in particular surface measure) due to lack of boundary regularity or  lack of strong convexity away from the axes. It also contains an example of a domain
in $\tilde \RR\setminus \RR$ with the property that surface measure is not admissible but measures of order $q$ are admissible when $|q|<1$.  In this case, $\Lop$ is not bounded on $L^2(\bndry D,\mu)$ for any rotation-invariant measure $\mu$.

In \S \ref{S:rr} we present the duality results mentioned earlier, and \S \ref{S:conc} contains a few concluding remarks.

\bigskip
 
\noindent{\bf Acknowledgment.} We are grateful to M. Lacey for helpful discussions, and to and E. M. Stein for 
raising the questions that have led to Examples 1 and 2
in \S 6.

\section{Geometric considerations}\label{S:geo}

Let $D\subset \C^2$ be a Reinhardt domain. Set
\begin{align}\label{E:gamma-def}
\gamma &=\gamma_D=\bndry D\cap \R_{\ge0}^2=\bndry D\cap \left( [0,\infty)\times[0,\infty) \right);\\
\label{E:gamma0-def}
\gamma_+ &=
 \, \bndry D\cap \R_{+}^2
=\left(\bndry D\setminus\{\zeta_1\zeta_2=0\}\right)\cap \R_{\ge0}^2\, .
\end{align}
(Here we are viewing $\R^2$ as a submanifold of $\C^2$.)

\begin{Prop}\label{P:transv}
In this situation, if $D$ has $C^k$-smooth boundary ($k\ge1$)  then the following will hold.

\refstepcounter{equation}\label{N:transv}
\begin{enum}
\item $i\R^2\subset T_\zeta\bndry D$ for each $\zeta\in\gamma$. \label{I:im}
\item $\bndry D$ meets $\R^2$ transversally. \label{I:transv}
\item $\gamma$ is a $C^k$-smooth 1-manifold. \label{I:1-mdf}
\item If $\zeta\in\gamma$ and $\zeta_1=0$ then $T_\zeta\bndry D=\C\times i\R$ and $T_\zeta \gamma = \R\times\{0\}$. \label{I:axis2}
\item If $\zeta\in\gamma$ and $\zeta_2=0$ then $T_\zeta\bndry D=i\R\times \C$  and $T_\zeta \gamma = \{0\}\times\R$.\label{I:axis1}
\end{enum}
\end{Prop}

\begin{proof}
If $\zeta_1\zeta_2\ne0$ then  
\itemref{N:transv}{I:im}  follows from from the fact that \[\{(e^{i\theta_1}\zeta_1,e^{i\theta_2}\zeta_2) \st \theta_1,\theta_2\in\R\}\subset\bndry D.\] The continuous dependence of $T_\zeta\bndry D$ on $\zeta$  now forces  \itemref{N:transv}{I:im} to hold also when $\zeta$ lies on one of the axes.

It follows now that $T_\zeta\bndry D + \R^2=\C^2$ for all $\zeta\in\gamma$ which shows that \itemref{N:transv}{I:transv} holds, and the transverse intersection theorem now implies \itemref{N:transv}{I:1-mdf}.

Item \itemref{N:transv}{I:axis2} follows from \itemref{N:transv}{I:im} and the invariance of $T_\zeta\bndry D$ under rotations in the $\zeta_1$ variable.  The proof of \itemref{N:transv}{I:axis1} is similar.
\end{proof}

\medskip

As in the introduction, we let $\tilde\RR$ denote the space of bounded convex complete $C^1$-smooth Reinhardt domains  in $\C^{2}$  that are $C^2$-smooth and strongly convex away from the axes $\{\zeta_{1}\zeta_{2}=0\}$.   Then $\gamma$ will be a $C^1$-smooth curve meeting both axes, while $\gamma_+$
will be $C^2$-smooth with non-vanishing curvature.  It follows easily that $\gamma$ will be the graph of a concave function, and in fact we easily verify the following.

\begin{Prop}\label{P:D-phi}
A Reinhardt domain $D$ belongs to $\tilde\RR$ if and only if it may be described as 
\begin{align}\label{E:DgenRein}
D&= \{(z_{1},z_{2})\st  |z_{2}|<\phi(|z_{1}|)\, ,\ |z_{1}|<b_1\}
\end{align}
where $b_1>0$ and $\phi$ is a continuous function on $[0, b_1]$ satisfying
\begin{subequations}\label{E:phi-cond}
\begin{align}
&\phi>0\text{ on } [0, b_1);\label{E:phi-pos}\\
&\phi(b_1)=0;\label{E:r-van}\\
&\phi' \text{ is continuous on $[0, b_1)$ and negative on $(0, b_1)$};\label{E:phi-dec}\\
&\phi'(0)=0;\label{E:horiz-phi}\\
&\phi'(t)\to-\infty\text{ as }t\to b_1;\label{E:vert-phi}\\
&\phi'' \text{ is continuous and negative on $(0, b_1)$}.\label{E:phi-concave}
\end{align}
\end{subequations}
\end{Prop}

Let $\rr$ be the map $\bndry D\to\gamma, (\zeta_1,\zeta_2)\mapsto (|\zeta_1|,|\zeta_2|)$.  Then any function $f$ on $\gamma$ induces a rotation-invariant function $f\circ \rr$ on $\bndry D$,  and every rotation-invariant function $f$ on $\bndry D$ may be recovered from its values on $\gamma$ by the formula
\begin{equation}\label{E:pb-gamma}
f=f\circ \rr.
\end{equation}

We will use $(r_1,r_2)$ as coordinates on $\R_{\ge0}^2$; thus \[\gamma=\{(r_1,r_2)\st 0\le r_1\le b_1, \, r_2=\phi(r_1)\}.\]
Extending these functions via \eqref{E:pb-gamma} we also have $r_j=|z_j|$ on $\bndry D$.

 Away from the axes, domains in $\tilde\RR$ are modeled after the $\PP$-domains described in \eqref{E:PPdom} and \eqref{E:Pball} in the following sense.

\begin{Prop}\label{P:osc} Suppose $D\in\tilde\RR$. Then, 
for every $\zeta\in\bndry D$ with $\zeta_{1}\zeta_{2}\ne 0$ there is a unique  $D_{p(\zeta),a_{1}(\zeta),a_{2}(\zeta)}\in\PP$ osculating $\bndry D$ to second order at $\zeta$.
\end{Prop}

\begin{proof} We start by considering points $\zeta=(r_1,r_2)\in\gamma_+$.  Noting that the curve $\bndry D_{p,a_{1},a_{2}}\cap\R_+^2$ is given by
$r_2=\sqrt[p]{\frac{1 - a_1r_1^p}{a_2}}$, we see that we need to determine  $p=p(\zeta)$, $a_{1}=a_1(\zeta)$ $a_2=a_2(\zeta)$ so that
\begin{align*}
\phi(r_1)&= \sqrt[\uproot{4}p]{\frac{1-a_{1}r_1^{p}}{a_{2}}}\\
\phi'(r_1)&= \frac{d}{dr_1}\sqrt[\uproot{4}p]{\frac{1-a_{1}r_1^{p}}{a_{2}}}
=-\frac{a_{1} r_1^{p-1} \left(\frac{1-a_{1}r_1^{p}}{a_{2}}\right)^{\frac{1}{p}-1}}{a_{2}} \\
\intertext{and}\\
\phi''(r_1)&= \frac{d^{2}}{dr_1^{2}}\sqrt[\uproot{4}p]{\frac{1-a_{1}r_1^{p}}{a_{2}}}
=\frac{a_{1} (1-p) r_1^{p-2}\left(\frac{1-a_{1}r_1^{p}}{a_{2}}\right)^{\frac{1}{p}}}{\left(1-a_{1}r_1^{p}\right)^{2}}.
\end{align*}

Substituting $1-a_{1}r_1^{p}=a_{2}\phi^{p}(r_1)$ throughout the second and third equations and solving for $a_{1}, a_{2}$ we obtain
\begin{align}
a_{1}(\zeta)&= \frac{(1-p)\left(\phi'(r_1)\right)^{2}}{r_1^{p}\phi(r_1)\phi''(r_1)}\label{E:a1def}\\
a_{2}(\zeta)&= \frac{(p-1)\phi'(r_1)}{r_1\phi^{p}(r_1)\phi''(r_1)}.\label{E:a2def}
\end{align}

Plugging these values back into the first equation and solving for $p$ we obtain 
\begin{equation}\label{E:p-one}
p(\zeta)\ =\ 1\ + \ \frac{r_1\phi(r_1)\phi''(r_1)}{\phi'(r_1)(\phi(r_1)-r_1\phi'(r_1))}.
\end{equation}

\medskip

Using \eqref{E:phi-cond} it is easy to check that $p(\zeta)>1$ 
and that $a_{1}(\zeta)$ and $a_{2}(\zeta)$ are positive.

We finish by extending $p, a_1$ and $a_2$ to functions on $\bndry D\setminus\{\zeta_1\zeta_2=0\}$ by setting $p=p\circ\rr, a_1=a_1\circ\rr$ and $a_2=a_2\circ\rr$ as in \eqref{E:pb-gamma}; rotation-invariance guarantees that the extended functions do what is required.
\end{proof}

Let $D$ be an  $\tilde\RR$-domain. Much of what we do below is made simpler by the introduction of  the following  auxiliary parameter on $\gamma_+$:

  \begin{equation}\label{E:subst2}
  s
     = \frac{-r_1\phi'(r_1)}{\phi(r_1)-r_1\phi'(r_1)}
   =\frac{-r_2\inv\,dr_2}{r_1\inv\,dr_1-r_2\inv\,dr_2}
   =1-\frac{r_1\inv\,dr_1}{r_1\inv\,dr_1-r_2\inv\,dr_2}.
  \end{equation}
  We note for later use that
  \begin{equation}\label{E:hidden-u}
\frac{dr_2}{r_2}= - \frac{s}{1-s} \frac{dr_1}{r_1}
\end{equation}
and
\begin{equation}\label{E:hidden-u-var}
\frac{dr_2}{dr_1}= - \frac{s}{1-s} \frac{r_2}{r_1}.
\end{equation}

  Our assumptions \eqref{E:phi-cond}  on $\phi$ 
  yield
  \begin{subequations}\label{E:lim-int} 
  \begin{align}
  &\quad\, s>0 \text{ on } \gamma_+\\ 
  &\lim\limits_{\zeta\to (0,b_2)} s(\zeta) = 0\, \label{E:s-a1}\\
   &\lim\limits_{\zeta\to (b_1,0)} s(\zeta) = 1\, \label{E:s-a2};
  \end{align}
  \end{subequations}
  moreover, differentiating \eqref{E:subst2} with respect to $r_1$ and using \eqref{E:p-one} we obtain
  \begin{equation}\label{E:ds/dt}
  \frac{ds}{dr_1} = \frac{sp}{r_1}.
  \end{equation}
  Thus $s$ is $C^1$-smooth on $\gamma_+$ and extends to a monotone continuous function (hence a homeomorphism) mapping $\gamma$ onto the interval $[0,1]$.

 Applying \eqref{E:subst2} to \eqref{E:ds/dt} we obtain the companion formula
  \begin{equation}\label{E:ds/dt-rev}
  \frac{d(1-s)}{dr_2} = \frac{(1-s)p}{r_2}.
  \end{equation}

The functions $s$ and $p$ determine the coordinate functions $r_1, r_2$ (up to multiplicative constants) as follows:
\begin{align*}
r_1(\zeta)=&b_1 \exp\left( -\int_\zeta^{(b_1,0)} \frac{ds}{sp}\right) 
\\
r_2(\zeta)=&b_2 \exp\left(\int_{(0,b_2)}^\zeta \frac{d(1-s)}{(1-s)p}\right). 
\end{align*}
(The integrals are taken over arcs of $\gamma$.)

Let 
\begin{equation}\label{E:pb-def}
\breve p=p\circ s\inv:(0,1)\to(1,\infty)
\end{equation}
i.e., $\breve p$ gives $p$ as a function of $s$. 
  Then we have
  \begin{subequations} \label{E:r-p-breve}
 \begin{align}
r_1 =&b_1 \exp\left( -\int_s^1 \frac{dt}{t\,\breve p(t)}\right) \label{E:r1-s}
\\
r_2 =&b_2 \exp\left(-\int_0^s \frac{dt}{(1-t)\,\breve p(t)}\right) \label{E:r2-s}
\end{align}
\end{subequations}
on $\gamma_+$
and so 
  \begin{multline}\label{E:s-param}
\gamma_+=\bigg\{\left(
b_1 \exp\left( -\int_s^1 \frac{dt}{t\,\breve p(t)}\right),
b_2 \exp\left(-\int_0^s \frac{dt}{(1-t)\,\breve p(t)}\right)
\right)\st \\
0<s<1\bigg\};
\end{multline}
thus also
 \begin{multline}\label{E:s-theta-param}
\bndry D\setminus \{\zeta_1\zeta_2=0\}\\
=\biggl\{\left(
b_1 \exp\left( -\int_s^1 \frac{dt}{t\,\breve p(t)}\right)e^{i\theta_1},
b_2 \exp\left(-\int_0^s \frac{dt}{(1-t)\,\breve p(t)}\right)e^{i\theta_2}
\right)\st \\
0<s<1, \theta_1\in[0,2\pi), \theta_2\in[0,2\pi)\biggr\}.
\end{multline}

\begin{Thm}\label{T:rr-tilde-cond}
The construction above defines a one-to-one correspondence between $\tilde\RR$ and the set of triples $\breve p, b_2, b_1$, where $b_2, b_1$ are positive constants and $\breve p:(0,1)\to(1,\infty)$ is a continuous function satisfying
\begin{subequations}\label{E:p-tildeRR-cond}
\begin{align}
\int_0^1 \frac{ds}{s\,\breve p(s)}&=\infty\label{E:a1-stretch}\\
\int_0^1 \frac{ds}{(1-s)\,\breve p(s)}&=\infty\label{E:a2-stretch}\\
\int_0^1 \frac{ds}{s\,\breve p^*(s)}&=\infty\label{E:horiz-pb}\\
\int_0^1 \frac{ds}{(1-s)\,\breve p^*(s)}&=\infty .\label{E:vert-pb}
\end{align}
\end{subequations}
Here, $\breve p^*(s)$ denotes the dual exponent of $\breve p(s)$ (that is $1/\breve p^*(s) + 1/\breve p(s) =1.$)
\end{Thm}

\begin{proof}
Suppose that $D\in\tilde\RR$.  Then condition \eqref{E:a1-stretch} follows from \eqref{E:s-a1} and \eqref{E:r1-s}.  Similarly, condition \eqref{E:a2-stretch} follows from \eqref{E:s-a2} and \eqref{E:r2-s}.
Next, we observe that \eqref{E:r1-s} yields
\begin{equation}\label{E:int1}
\int_s^1 \frac{dt}{t\,\breve p^*(t)} =
-\log s + \log (r_1/b_1).
\end{equation}
Moreover, conditions  \eqref{E:horiz-phi}  and
\eqref{E:hidden-u-var} 
imply
\begin{equation}\label{E:lims/r}
\frac{s}{r_1}\to 0 \quad \mathrm{as}\ s\to 0,
\end{equation}
so that \eqref{E:horiz-pb} follows from \eqref{E:int1} and \eqref{E:lims/r}.
Identity  \eqref{E:vert-pb} follows by a parallel argument.

Suppose now that we are given positive constants $b_1, b_2$ together with a continuous function $\breve p$ satisfying \eqref{E:a1-stretch} through \eqref{E:vert-pb}. Then \eqref{E:s-param} describes an open arc $\gamma_+$ in $\R_+^2$, and conditions \eqref{E:a1-stretch} and \eqref{E:a2-stretch} imply that $\gamma_+$ extends to a closed arc $\gamma$ in $\R_{\ge0}^2$ with endpoints at $(0, b_2)$ and $(b_1,0)$.  The monotonicity of the resulting $r_1$ and $r_2$ as functions of $s$ (see \eqref{E:r1-s}, \eqref{E:r2-s}) shows that $\gamma$ is the graph of a continuous decreasing function $\phi$ on $[0,b_1]$ satisfying \eqref{E:phi-pos} and \eqref{E:r-van}.  Moreover, using \eqref{E:r1-s} and \eqref{E:r2-s} we find that
\begin{equation}\label{E:phi-diff}
\phi'(r_1) = \frac{dr_2/ds}{dr_1/ds}=-\frac{s}{1-s}\frac{r_2}{r_1}
\end{equation}
is continuous and negative on $(0,b_1)$.
 Taking
\eqref{E:int1} into account, we find that  condition \eqref{E:horiz-pb} implies 
\eqref{E:lims/r}, and using 
 \eqref{E:phi-diff} we see that $\phi'(r_1)\to 0$ as $r_1\to 0$; thus we have verified \eqref{E:phi-dec} and \eqref{E:horiz-phi}.  A similar argument allows us to deduce \eqref{E:vert-phi} from \eqref{E:vert-pb}.  Finally, using
\begin{equation*}
\frac{ds}{dr_1}=\frac{1}{dr_1/ds}=\frac{s\breve p(s)}{r_1}
\end{equation*}
to differentiate
 \eqref{E:phi-diff} we find that $\phi''(r_1)=-\frac{(\breve p(s)-1)s}{(1-s)^2}\frac{r_2}{r_1^2}$,   verifying \eqref{E:phi-concave}. We have shown that $\phi$ satisfies the conditions of Proposition \ref{P:D-phi}, thus we may use \eqref{E:DgenRein} to define the desired domain $D\in\tilde\RR$.
\end{proof}

\begin{Def}\label{D:gen}
We refer to the domain $D$ constructed at the end of the previous proof as the domain {\em generated by} $\breve p,b_2, b_1$.
\end{Def}
\begin{Remark}\label{R: ext-par} The parameterizations \eqref{E:s-param} and \eqref{E:s-theta-param} extend to parameterizations of all of $\gamma$ and $\bndry D$, respectively, with $s$ ranging over the closed interval $[0,1]$.
\end{Remark}

\begin{Lem}\label{L:HB}
For $D\in \tilde\RR$, $(w_1,w_2)\in \bar D$, $(r_1,r_2)\in\gamma_+$ with $(|w_1|,|w_2|)\ne(r_1,r_2)$ we have 
\begin{equation}\label{E:HB-s}
\frac{s}{r_1}|w_1|+\frac{1-s}{r_2}|w_2|<1.
\end{equation}
\end{Lem}

\begin{proof}
The strict convexity of $\bar D\cap \R_+^2$ implies that $(|w_1|,|w_2|)$
lies below the tangent line $x_2=r_2+\phi'(r_1)(x_1-r_1)$ to $\gamma_+$ at $(r_1,r_2)$, that is,
\begin{equation}\label{E:HB}
|w_2|-r_2< \phi'(r_1)(|w_1|-r_1).
\end{equation}
Combining this with $\frac{s}{r_1} =  
\frac{-\phi'(r_1)}{r_2 -r_1\phi'(r_1)}$ and  $\frac{1-s}{r_2}  = \
\frac{1}{r_2 -r_1\phi'(r_1)}$, see \eqref{E:subst2}, we obtain \eqref{E:HB-s}.
\end{proof}

\begin{Lem}\label{L:finite-coeff}
For  $D\in \tilde\RR$  
we have
\begin{equation}\label{E:finite-coeff}
\frac{s}{r_1}\leq \frac{1}{b_1}
\text{ and }\frac{1-s}{r_2}\leq \frac{1}{b_2}\
\end{equation}
on $\gamma_+$.
\end{Lem}

\begin{proof} This follows from \eqref{E:HB-s} by setting $w=(b_1,0)$ and $(0,b_2)$, respectively. 
\end{proof}

\begin{Lem}\label{L:squ}
For  $D\in \tilde\RR$ the functions $\dfrac{s}{r_1}$ and $\dfrac{1-s}{r_2}$ extend to continuous functions on $\gamma$, and the function
$\left(\dfrac{s}{r_1}\right)^2+\left(\dfrac{1-s}{r_2}\right)^2$ extends to a continuous positive function on $\gamma$.
\end{Lem}

\begin{proof}
This is a consequence of the limits $\lim\limits_{\zeta\to(0,b_2)} \dfrac{s}{r_1}=0$, $\lim\limits_{\zeta\to(b_1,0)} \dfrac{s}{r_1}=\dfrac{1}{b_1}$, $\lim\limits_{\zeta\to(0,b_2)} \dfrac{1-s}{r_2}=\dfrac{1}{b_2}$, $\lim\limits_{\zeta\to(b_1,0)} \dfrac{1-s}{r_2}=0$.
(See \eqref{E:hidden-u-var} to check the first and fourth limits.)
\end{proof}

In the case of a $\PP$-domain the $s$-parametrization
of $\gamma$ given in 
 \eqref{E:s-param} takes the following especially simple form: 
  \begin{equation}\label{E:gamma-PP}
 \gamma =\{(b_1s^{1/p},\, b_2(1-s)^{1/p})\st 0\leq s\leq 1\}.
 \end{equation}

\begin{Prop}\label{P:pcon} Suppose $D\in\tilde\RR$. If the function $p$  is constant then $D\in\PP$.
\end{Prop}

\begin{proof}
If $p$ is constant then \eqref{E:s-param} matches \eqref{E:gamma-PP}.
\end{proof}

For general $D\in \tilde\RR$ there will be no control on the behavior of $p$ along $\gamma_+$ as we approach one of the endpoints, so we will also consider the following smaller class of domains.

\begin{Def}\label{D:rr}
Let $\RR$ denote the class of domains
\begin{equation*}
\{(z_{1},z_{2})\st |z_{1}|<b_1, \,|z_{2}|<\phi(|z_{1}|)\}
\end{equation*}
with $b_1$ a given positive constant  and $\phi$ a continuous decreasing concave function on $[0,b_1]$ which  is $C^2$-smooth on $(0,b_1)$ and satisfies
\begin{subequations}\label{E:phi-rr-cond}
\begin{align}
\phi''(r_1)&<0\quad \text{for } 0<r_1<b_1;\label{E:rrcx}\\
\phi(r_1)&=b_2-c_2r_1^{p_{1}}+\eps_1(r_1) \text{ for  $r_1$ near $0$};\label{E:rrz1}\\
\phi(r_1)&=\sqrt[\uproot{4}p_{2}]{\frac{b_1-r_1+\eps_2(\phi(r_1))}{c_1}}\quad \text{ for  $r_1$ near $b_1$}\label{E:rrz2}
\end{align}
\end{subequations}
where $b_j>0,$  $c_j>0$ and $p_j>1$ are  constants and $\eps_j(r_j)$ are  functions satisfying 
\begin{subequations}\label{E:eps-cond}
\begin{align}
&\text{$\eps_j$ is of class $C^1$ for $r_1\ge 0$};\label{E:eps-a}\\
&\text{$\eps_j$ is of class $C^2$ for $r_1>0$};\\
&\eps_j(0)=0;\\
 &\eps_j'(0)=0;\label{E:eps-d}\\
 & \eps''_j(r_j)=o(r_j^{p_j-2})\label{E:eps-e}
\end{align}
\end{subequations}
for $j=1, 2$.
\end{Def}

The conditions \eqref{E:phi-rr-cond}  imply the conditions \eqref{E:phi-cond} and so $\RR$ is contained in $\tilde\RR$.

Condition \eqref{E:rrz2} is equivalent to the condition that $\psi=\phi\inv$ satisfies
\begin{equation}\label{E:rrz1inv}
\psi(r_2)=b_1-c_1r_2^{p_{2}}+\eps_2(r_2) \text{ for  $r_2$ near $0$}.
\end{equation}

The class $\RR$ is invariant under permutation of the coordinates $z_{1},z_{2}$; thus we will often transfer work on behavior near the axis $z_{1}=0$ to get  corresponding results near $z_{2}=0$.

Note that the assumptions \eqref{E:eps-cond} imply that 
\begin{subequations}\label{E:eps-est}
\begin{align}
\eps'_1(r_1)&=o(r_1^{p_1-1})\label{E:eps-est-a}\\
\eps_1(r_1)&=o(r_1^{p_1}).
\end{align}
\end{subequations}

As mentioned in the introduction,  the class $\RR$ contains  the $\PP$-domains \eqref{E:PPdom}.  For a $\PP$-domain, the constants in \eqref{E:rrcx}, \eqref{E:rrz1} and \eqref{E:rrz2} are  determined in terms of $p$ and $a_1$, $a_2$ by
\begin{equation*}
p_1=p_2=p, \;
  b_2=\frac{1}{\sqrt[p]{a_2}},\;
  b_1=\frac{1}{\sqrt[p]{a_1}}, \;
c_2=\frac{a_1}{p\sqrt[p]{a_2}},\;
 \quad c_1=\frac{a_2}{p\sqrt[p]{a_1}},
 \end{equation*}
the function $\eps_1(|\zeta_1|)$ is the error term of the first-order expansion 
of
\begin{equation}\label{E:phi-pp}
\phi(|\zeta_1|) =\sqrt[p]{\frac{1 - a_1|\zeta_1|^p}{a_2}}\, =\,
b_2\sqrt[p]{1-b_1^{-p}|\zeta_1|^p\,}
\end{equation}
in powers of $|\zeta_1|^p$ about $\zeta_1=0$, while $\eps_2(|\zeta_2|)$ is similarly determined by
\begin{equation*}\label{E:psi-pp}
\psi(|\zeta_2|) =\sqrt[p]{\frac{1 - a_2|\zeta_2|^p}{a_1}}\,=\,
\, b_1\sqrt[p]{1-b_2^{-p}|\zeta_2|^p\,} .
\end{equation*}

\begin{Prop}\label{P:RRdom}
Suppose $D\in\RR$. Then the functions $p(\zeta)$, $a_1(\zeta)$ and $a_2(\zeta)$ described in Proposition \ref {P:osc} extend to continuous functions on all of $\bndry D$ with $p(\zeta)=p_1$ when $\zeta_1=0$ and $p(\zeta)=p_2$ when $\zeta_2=0$.
\end{Prop}

\begin{proof} Using \eqref{E:pb-gamma} as before it will suffice to show that the functions $p, a_1$ and $a_2$ extend continuously from $\gamma_+$ to $\gamma$ (with $p$ taking the indicated boundary values).

 Combining \eqref{E:eps-est-a}
with \eqref{E:rrz1} and \eqref{E:subst2} we find that
\begin{align}\label{s-asymp}
s&=\frac{p_1c_2r_1^{p_1}-r_1\eps'_1(r)}
{b_2+(p_1-1)c_2r_1^{p_1} + \eps_1(r_1)-r_1\eps'(r_1)}\\
&=\frac{p_1c_2}{b_2} r_1^{p_1} + o(r_1^{p_1})\notag
\end{align}
and
\begin{equation}\label{ds-asymp}
\frac{ds}{dr_1}=\frac{p_1^2c_2}{b_2} r_1^{p_1-1}+ o(r_1^{p_1-1}).
\end{equation}

We note for future reference that \eqref{s-asymp} may be rewritten in the form
\begin{equation}\label{E:r1-asympt}
r_1=s^{1/p_1}\left( \left(\frac{b_2}{p_1 c_2}\right)^{1/p_1}+o(1)\right).
\end{equation}

From \eqref{E:ds/dt} we now obtain

\begin{align*}
p(\zeta)&=\frac{r_1}{s}\frac{ds}{dr_1}=p_1+o(1)
\end{align*}
as $\zeta\to (0,b_2)$.

Applying a similar analysis to \eqref{E:a1def} and \eqref{E:a2def} we find that 
$
a_{1}(\zeta)\to \frac{p_{1}c_2}{b_2}
$
and
$
a_{2}(\zeta)\to \frac{1}{b_2^{p}}
$
as $\zeta\to (0,b_2)$.

 Transferring these results to the other axis we have
 \begin{align}\label{(1-s)-asymp}
1-s&=\frac{p_2c_1}{b_1} r_2^{p_2} + o(r_2^{p_2});
\end{align}
also
$
p(\zeta)\to p_{2},\,
a_{2}(\zeta)\to \frac{p_{2}c_1}{b_1}
$
and
$
a_{1}(\zeta)\to \frac{1}{b_1^{p}}
$
as $\zeta\to (b_1,0)$.
\medskip
\end{proof}

\begin{Cor}\label{P:smdom}
Suppose $D$ is a $C^2$-smooth strongly convex Reinhardt domain in $\C^2$. Then the function $p$ defined by \eqref{E:p-one} extends to a continuous rotation-invariant function on $\bndry D$ satisfying $p(\zeta)=2$ when $\zeta_1\zeta_2=0$.\end{Cor}

\begin{proof}
Such a domain satisfies Definition \ref{D:rr} with $p_1=p_2=2$.
\end{proof}

\begin{Thm}\label{T:p-rr-cond}
A domain generated by $\breve p$, $b_2$ and $b_1$ as in \eqref {E:s-param}
belongs to $\RR$ if and only if $\breve p$ satisfies the conditions
\begin{subequations}\label{E:p-rr-cond}
\begin{align}
&\text{$\breve p$ extends to a continuous function $[0,1]\to (1,\infty)$};
\label{E:p-ext}\\
&\int_0^1\left(\frac{1}{\breve p(s)}- \frac{1}{\breve p(0)}\right)\frac{ds}{s}
\text{ and }
\int_0^1\left(\frac{1}{\breve p(s)}- \frac{1}{\breve p(1)}\right)\frac{ds}{1-s}
\label{E:p-Dini} \\
&\qquad\qquad\qquad\qquad\qquad\qquad\text{ converge as improper integrals.}\notag
\end{align}
\end{subequations}
(The condition \eqref{E:p-Dini} means that $\lim\limits_{s\to 0+} \int_s^1\left(\frac{1}{\breve p(t)}- \frac{1}{\breve p(0)}\right)\frac{dt}{t}$ and \linebreak $\lim\limits_{s\to 1-}\int_0^s\left(\frac{1}{\breve p(t)}- \frac{1}{\breve p(1)}\right)\frac{dt}{1-t}$ exist and are finite.)
\end{Thm}

Note that \eqref{E:p-ext}
implies \eqref{E:a1-stretch}-\eqref{E:vert-pb}.

\begin{proof} 
Suppose our domain is in $\RR$.  Then Proposition \ref{P:RRdom} shows that 
\eqref {E:p-ext} holds with $\breve p(0)=p_1, \breve p(1)=p_2$.  
Combining this with \eqref{E:r-p-breve} we obtain
\begin{align}\label{E:dini}
\int_s^1\left(\frac{1}{\breve p(t)}- \frac{1}{\breve p(0)}\right)\frac{dt}{t}\notag
&=-\log\frac{r_1}{b_1}+\frac{1}{p_1}\log s\\
&=\frac{1}{p_1}\log\frac{s b_1^{p_1}}{r_1^{p_1}}.\
\end{align}

Furthermore,  \eqref{s-asymp}  guarantees that the expression above converges to $\frac{1}{p_1}\log\frac{p_1 c_2 b_1^{p_1}}{b_2}$.
Then
a similar argument establishes the other half of \eqref{E:p-Dini}.

Suppose now that the conditions \eqref{E:p-rr-cond} hold.  Note that \eqref{E:p-ext}
implies the conditions in \eqref{E:p-tildeRR-cond}, so $D\in\tilde \RR$.  We need to specify constants $b_2, p_1, p_2, c_1$ and $ c_2$ so that 
all the conditions of Definition \ref{D:rr} hold.  We set $p_1=\breve p(0)$, $p_2=\breve p(1)$ and $b_2=\phi(0)$.

For any $c_2>0$ we find that $\eps_1$ defined from \eqref{E:rrz1}, that is, \begin{equation}\label{E:eps1-def}
\eps_1(r_1)=r_2-b_2+c_2r_1^{p_{1}},
\end{equation}
satisfies conditions \eqref{E:eps-a} through \eqref{E:eps-d} for $j=1$, and we are left to determine $c_2$ so that \eqref{E:eps-e} is satisfied. 
We set
\begin{equation*}
c_2=\frac{b_2}{p_1b_1^{p_1}} \exp\left( p_1\int_0^1\left(\frac{1}{\breve p(s)}-\frac{1}{p_1}\right)\,\frac{ds}{s}\right)
\end{equation*}
Since \eqref{E:dini} holds as before, 
we find that 
\begin{equation}\label{E:sr1p1}
s= \frac{p_1c_2}{b_2} r_1^{p_1} + o(r_1^{p_1}).
\end{equation}

Differentiating \eqref{E:eps1-def} twice with the use of 
 \eqref{E:hidden-u-var} and \eqref{E:ds/dt}  we obtain that
\begin{align*}
\eps'_1(r_1)&=-\frac{s}{1-s}\frac{r_2}{r_1}+c_2p_1r_1^{p_1-1},\\
\eps''_1(r_1)&=-(p-1)\frac{s r_2}{(1-s)^2r_1^2}+(p_1-1)c_2p_1 r_1^{p_1-2}.
\end{align*} 
Combining these with \eqref{E:sr1p1} and $r_2=b_2+o(1)$,  $p=p_1+o(1)$ we find that \eqref{E:eps-e} holds for $j=1$. 

A similar argument takes care of $j=2$.
\end{proof}

\section{Construction and basic properties of the Leray transform for domains in the class $\tilde \RR$}\label{S:Lt}

In this section we 
compute the Leray kernel for domains in the class $\tilde\RR$ and check that the associated Leray transform $\Lop$ reproduces holomorphic functions from their boundary values.  We
introduce the notion of admissible measure and
 provide formulae for various norms and spectra. (Unless explicitly stated, at this stage $\Lop$ is not assumed to be $L^2$-bounded.)

We base our computations on the function 
\begin{equation}\label{E:rhoRR}
\rho(\zeta_1, \zeta_2) \,=\, |\zeta_2| - \phi(|\zeta_1|)
\end{equation}
where $\phi$ is as in \eqref{E:phi-cond}.  This function will fail to be differentiable at points where $\zeta_1\zeta_2=0$; moreover, 
\begin{equation}\label{E:phi-grad}
|\nabla \rho(\zeta)|=\sqrt{1+\left(\phi'(|\zeta_1|)\right)^2},
\end{equation}
will not be bounded above where defined.  So $\rho$ is a defining function for $\bndry D\setminus\{\zeta_1\zeta_2=0\}$, but not for $\bndry D$.

For $w\in D$, \eqref{E:Lker} still defines a three-form on $\bndry D\setminus\{\zeta_1\zeta_2=0\}$ which is independent of the particular choice of defining function.  When integrating expressions involving this form over $\bndry D$ we simply ignore the points where $\zeta_1\zeta_2=0$.  (The set of such points has measure zero with respect to all boundary measures considered below.)

\bigskip

The classical proof (see \cite{Ran},  \S IV.3.2) of the reproducing property for holomorphic functions no longer applies, but we remedy this in 
Corollary \ref{C:poly-rep} and Proposition \ref{P:Hardy} below.

\begin{Lem}\label{L:Lker}
Let $D\in \tilde \RR$.  Then, representing $\zeta\in\bndry D\setminus\{\zeta_1\zeta_2=0\}$ by the coordinates $(s,\theta_1,\theta_2)$ as in 
\eqref{E:s-theta-param}, we have
\begin{equation}\label{E:LkerRR}
L(\zeta,w)= \frac{ds\w d\theta_1\w d\theta_2}
{4\pi^2
\left(1-e^{-i\theta_1}\frac{s}{r_1}w_1-e^{-i\theta_2}\frac{1-s}{r_2}w_2\right)^2}.
\end{equation}
\end{Lem}

\begin{proof}
From \eqref{E:subst2}, \eqref{E:p-one}, \eqref{E:ds/dt} and \eqref{E:ds/dt-rev} we obtain $\phi'(r_1)=-\frac{s}{1-s}\frac{r_2}{r_1}$, 
$\phi''(r_1)=-\frac{(p-1)s}{(1-s)^2}\frac{r_2}{r_1^2}$, $dr_1=\frac{r_1}{p}\frac{ds}{s}$ and $dr_2=-\frac{r_2}{p}\frac{ds}{1-s}$.

Using \eqref{E:Lker} to compute $L(\zeta,w)$ we first compute
$\dee\rho\w\deebar\dee\rho$ with $\rho$ as in \eqref{E:rhoRR};
 then, setting $z_j=r_je^{i\theta_j}$ and applying the above formulae we obtain
\begin{equation}\label{E:leray-num}
j^*(\dee\rho\w\deebar\dee\rho) = -\frac{r_2^2}{4(1-s)^{2}}\,{ds\w d\theta_1\w d\theta_2}.
\end{equation}
Turning our attention to the denominator we find that
\begin{multline*}
\left(\dee\rho(\zeta)\bullet(\zeta-w)\right)^2 \\= 
\left(
\frac12 e^{-i\theta_1}\frac{s}{1-s}\frac{r_2}{r_1}(r_1e^{i\theta_1}-w_1)
+\frac12 e^{-i\theta_2}(r_2e^{i\theta_2}-w_2)
\right)^2.
\end{multline*}
Dividing and simplifying we obtain \eqref{E:LkerRR}.
\end{proof}

From \eqref{E:HB-s} we have $\left|e^{-i\theta_1}\frac{s}{r_1}w_1+e^{-i\theta_2}\frac{1-s}{r_2}w_2\right|<1$.  Thus by the differentiated geometric series we have
\begin{equation*}
L(\zeta,w)= \frac{ds\w d\theta_1\w d\theta_2}{4\pi^2} \sum\limits_{j=0}^\infty
(j+1) \left(e^{-i\theta_1}\frac{s}{r_1}w_1+e^{-i\theta_2}\frac{1-s}{r_2}w_2\right)^j.
\end{equation*}
Using the binomial theorem we obtain the following result.

\begin{Lem}\label{L:Lker-series} The Leray kernel admits the expansion
\begin{multline*}
L(\zeta,w)\\
= \frac{ds\w d\theta_1\w d\theta_2}{4\pi^2} \sum\limits_{n,m=0}^\infty
\frac{(n+m+1)!}{n! \,m!} \left(\frac{s}{r_1}\right)^n \left(\frac{1-s}{r_2}\right)^m\\
\cdot
w_1^n w_2^m e^{-i(n\theta_1+m\theta_2)}
\end{multline*}
converging uniformly (with exponential speed) for $w$ in any compact subset of $D$. (In fact, the convergence is uniform on compact subsets of $\bar D\setminus\{\zeta\}$.)
\end{Lem}

\begin{Def}
We say that a function $f$ on $\bndry D$ is an {\em $(n,m)$-monomial} if it takes the form  
\begin{equation}\label{E:mono-def}
f(\zeta)=g(s)e^{i(n\theta_1+m\theta_2)}.
\end{equation}
\end{Def}

\begin{Cor}\label{C:L-mono}
If $f$ is an  $(n,m)$-monomial of the form \eqref{E:mono-def} then for $w\in D$ we have
\begin{equation}\label{E:L-mono}
\Lop f(w)
=\begin{cases}
\quad 0 & \qquad\qquad\qquad\qquad\qquad\text{ if }\min\{n, m\} <0;\\
\frac{(n+m+1)!}{n! \,m!}
&\left(
{\displaystyle \int_0^{1}} 
g(s) \left(\dfrac{s}{r_1}\right)^n \left(\dfrac{1-s}{r_2}\right)^m \,ds \right)
w_1^n w_2^m \\
&\qquad\qquad\qquad\qquad\qquad\text{ if }\min\{n, m\}\ge 0.
\end{cases}
\end{equation}
\end{Cor}

\begin{proof}
This follows from Lemma \ref{L:Lker-series} (or Lemma \ref{L:Lker} and a residue computation).
\end{proof}

If $f(\zeta)=\zeta_1^n\zeta_2^m$ then applying Corollary \ref{C:L-mono} with $g(s)=r_1^n r_2^m$ 
and recalling that
\begin{equation}\label{E:beta}
\int_0^1 s^n(1-s)^m\,ds=\frac{n! \,m!}{(n+m+1)!}
\end{equation}
we find that $\Lop f(w)=w_1^n w_2^m$ for $w\in D$.  Taking sums we obtain the following.

\begin{Cor}\label{C:poly-rep}
The operator $\Lop$ reproduces holomorphic polynomials from their restrictions to $\bndry D$.
\end{Cor}

Returning to Corollary \ref{C:L-mono} we see that when $f$ is an $(n,m)$-monomial $g(s)e^{i(n\theta_1+m\theta_2)}$ then $\Lop f$ extends continuously to $\bar D$ with boundary values given (in the non-trivial cases) by 
\begin{multline}\label{E:mono-bndry}
\Lop f(R_1e^{i\theta_1},R_2e^{i\theta_2})\\=\frac{(n+m+1)!}{n! \,m!}
\Bigg(\int_0^{1} 
g( s) \left(\dfrac{ s}{ r_1}\right)^n \left(\dfrac{1- s}{ r_2}\right)^m \,d s
\Bigg)
R_1^{^{\, n}}\, R_2^{^{\,m }}e^{^{i(n\theta_1+m\theta_2)}}.
\end{multline}
In particular, $\Lop$ maps $(n,m)$-monomials to $(n,m)$-monomials.

\medskip

Let $\mu$ be a rotation-invariant measure on $\bndry D$ described by
\begin{equation}\label{E:meas1}
d\mu=\frac{1}{4\pi^2}\omega(s)\,ds\,d\theta_1\,d\theta_2,
\end{equation}
where $\omega(s)$ is measurable and positive a.e.

\begin{Def}\label{D:Four-piece}
Let $L^2_{n,m}(\bndry D,\mu)$ denote the space of $(n,m)$-monomials \eqref{E:mono-def} that are in $L^2(\bndry D,\mu)$.
\end{Def}

The spaces $L^2_{n_1,m_1}(\bndry D,\mu)$ and $L^2_{n_2,m_2}(\bndry D,\mu)$ are orthogonal subspaces of $L^2(\bndry D,\mu)$ unless $(n_1,m_1)=(n_2,m_2)$.
Note also that if $f_1(\zeta)=g_1(s)e^{i(n\theta_1+m\theta_2)}$ and $f_2(\zeta)=g_2(s)e^{i(n\theta_1+m\theta_2)}$ are in $L^2_{n,m}(\bndry D,\mu)$ then the Hermitian inner product $\langle f_1, f_2\rangle$ of the monomials in $L^2_{n,m}(\bndry D,\mu)$ is just \linebreak $\int_0^1 g_1(s) \bar{g_2(s)} \omega(s)\, ds.$

\begin{Prop} \label{P:piece-norm}
When $n,m\ge 0$, the restriction $\Lop_{n,m}$ of $\Lop$ to \linebreak
$L^2_{n,m}(\bndry D,\mu)$ is a rank-one projection operator with $L^2$ operator norm given by
\begin{multline}\label{E:piece-norm}
\|\Lop_{n,m}\|_\mu^2\\
=
\left(\frac{(n+m+1)!}{n! \,m!}\right)^2
\int_0^{1} \left(\dfrac{  s}{r_1}\right)^{2n} \left(\dfrac{1- s}{ r_2}\right)^{2m}\frac{1}{\omega(s)}\,ds\\
\cdot\int_0^{1} r_1^{2n} r_2^{2m}\omega(s)\,ds.
\end{multline}
\end{Prop}

\begin{proof}
Set 
\begin{align}\label{E:kappa-tau}
\kappa_{n,m}&=\frac{(n+m+1)!}{n! \,m!}
\left(\dfrac{  s}{r_1}\right)^{n}\left(\dfrac{1- s}{ r_2}\right)^{m}
\frac{1}{\omega(s)}e^{i(n\theta_1+m\theta_2)},\\
 \tau_{n,m}&=r_1^{n}r_2^{m}e^{i(n\theta_1+m\theta_2)}.\notag
\end{align}
Then from \eqref{E:mono-bndry} and Corollary \ref{C:L-mono} and using the formula above for the inner product  in $L_{n,m}^2(\bndry D, \mu)$ we have
\begin{equation}\label{E:kaptau}
\Lop_{n,m}(f)=\langle f,\kappa_{n,m}\rangle\,\tau_{n,m}
\end{equation}
and \eqref{E:beta} yields
\begin{equation}\label{E:kt1}
\langle \tau_{n,m},\kappa_{n,m}\rangle=1
\end{equation}
so that
\begin{equation}\label{E:lmn2}
\Lop_{n,m}^2=\Lop_{n,m}
\end{equation}
and 
\begin{equation}\label{E:lkt}
\|\Lop_{n,m}\|_\mu^2
=
\|\kappa_{n,m}\|_\mu^2\,\,\|\tau_{n,m}\|_\mu^2\, .
\end{equation}
\end{proof}

\begin{Thm}\label{T:adm}
Let $D\in\tilde\RR$ and let $\mu$ be a rotation-invariant measure on $\bndry D$ described by
\eqref{E:meas1}
with $\omega(s)$  measurable and positive a.e. Then the following conditions are equivalent.
\refstepcounter{equation}\label{N:adm}
\begin{enum}
\item $\int_0^{1}\frac{1}{\omega(s)}\,ds$ and $\int_0^{1} \omega(s)\,ds$ are finite. \label{I:adm}
\item The measure $\mu$ is finite, and the functional $f\mapsto (\Lop f)(0)$ is bounded on $L^2(\bndry D,\mu)$. \label{I:w=0}
\item The measure $\mu$ is finite, and for each $w\in D$ the  functional $f\mapsto (\Lop f)(w)$ is bounded on $L^2(\bndry D,\mu)$. \label{I:w-gen}
\item $\|\Lop_{0,0}\|_\mu<\infty$. \label{I:n=m=0}
\item For each $(n,m)$ we have $\|\Lop_{n,m}\|_\mu<\infty$. \label{I:(n,m)-gen}
\end{enum}
\end{Thm}

\begin{proof}
The equivalence of items \itemref{N:adm}{I:adm} and \itemref{N:adm}{I:n=m=0} is immediate from \eqref{E:piece-norm}. Similarly, \eqref{E:piece-norm}
 together with Lemma \ref{L:squ} and the boundedness of $r_1$ and $r_2$
 show in turn  that items \itemref{N:adm}{I:n=m=0} and\itemref{N:adm}{I:(n,m)-gen} are equivalent.

To see that items \itemref{N:adm}{I:w=0} and \itemref{N:adm}{I:adm} are equivalent, note that \[\mu(\bndry D)=\int_0^{1} \omega(s)\,ds\] and that from \eqref {E:LkerRR}   and \eqref{E:Lop} we have that \[\Lop f(0)=\int\limits_{\bndry D} \frac{1}{\omega(s)}f(\zeta)\,d\mu = \left<f, 1/\omega\right>\] is a linear functional on $L^2(\bndry D,\mu)$ with norm $\sqrt{\int_0^{1}\frac{1}{\omega(s)}\,ds}$.

Finally, to see that items \itemref{N:adm}{I:w=0} and \itemref{N:adm}{I:w-gen} are equivalent, fix $w\in D$ and  note that
Lemmas \ref{L:HB} and \ref{L:squ} allow us to conclude  that 
\begin{equation*}
\sup\limits_{\zeta=\left(r_1e^{i\theta_1},r_2e^{i\theta_2}\right)\in\bndry D} \left|e^{-i\theta_1}\frac{s}{r_1}w_1+e^{-i\theta_2}\frac{1-s}{r_2}w_2\right|\le1.
\end{equation*}
Consulting \eqref{E:LkerRR} we see that
\begin{equation*}
\frac{1}{2}\leq 
\inf\limits_{\zeta\in\bndry D} \left| \frac{L(\zeta,w)}{L(\zeta,0)} \right| \le  \sup\limits_{\zeta\in\bndry D} \left| \frac{L(\zeta,w)}{L(\zeta,0)}\right| <\infty, 
\end{equation*}
from which the desired result follows immediately.
\end{proof}

\begin{Def}\label{D:adm}
We will call a measure of the form \eqref{E:meas1} {\em admissible} if it satisfies the equivalent conditions of Theorem \ref {T:adm}.
\end{Def}

\begin{Remark}\label{R:need-adm}
It is easy to see that any rotation-invariant measure $\mu$ on $\bndry D$ satisfying condition \itemref{N:adm}{I:w=0} in Theorem \ref{T:adm}
must in fact be of the form 
\eqref{E:meas1}
with $\omega(s)$  measurable and positive a.e.
\end{Remark}

Assume now that $\mu$ is admissible. Because $D$ is Reinhardt, any $f\in L^2(\bndry D,\mu)$ may be written uniquely as a sum
$f=\sum\limits_{n,m\in\Z} f_{n,m}$ converging in $L^2(\bndry D,\mu)$ where each $f_{n,m}$ is an $(n,m)$-monomial \eqref{E:mono-def} and \[\|f\|^2_\mu = \sum\limits_{n,m\in\Z}\|f_{n,m}\|^2_\mu.\]
If $\Lop$ is to define a bounded operator on $L^2(\bndry D,\mu)$ it must be given by
\begin{equation}\label{E:L-Four-def}
\Lop f = \sum_{n,m\ge 0} \Lop_{n,m} f_{n,m}
\end{equation}
and thus
\begin{equation*}
\|\Lop f\|^2_\mu = \sum_{n,m\ge 0} \|\Lop_{n,m} f_{n,m}\|^2_\mu.
\end{equation*}
From this we easily obtain the following.

\begin{Thm}\label{T:Lnrm}
$\Lop$ defines a bounded operator on $L^2(\bndry D,\mu)$ if and only if the quantities $\|\Lop_{n,m}\|_\mu$ given in \eqref{E:piece-norm} are uniformly bounded for ${n,m\ge 0}$; moreover,
\begin{equation*}
\|\Lop \|_\mu = \sup\, \{\|\Lop_{n,m} \|_\mu\st {n,m\ge 0} \}.
\end{equation*}
\end{Thm}

Condition \itemref{N:adm}{I:(n,m)-gen} in Theorem \ref{T:adm} shows that
when $\mu$ is admissible then  the boundary values of holomorphic polynomials lie in $L^2(\bndry D,\mu)$.  This observation motivates the following.

\begin{Def}
 The Hardy space $H^2(\bndry D,\mu)$ is the closure  in $L^2(\bndry D,\mu)$ of the boundary values of holomorphic polynomials.
\end{Def}

  From Corollaries \ref{C:L-mono} and  \ref{C:poly-rep} and Proposition \ref{P:piece-norm} we obtain the following.

\begin{Prop}\label{P:Hardy}
If $\Lop$ defines a bounded operator on $L^2(\bndry D,\mu)$ then $\Lop$ is a projection operator from $L^2(\bndry D,\mu)$ onto $H^2(\bndry D,\mu)$.
\end{Prop}

When $\Lop$  defines a bounded operator on $L^2(\bndry D,\mu)$ 
then $\Lop$ admits an adjoint  $\Lop^*_\mu$.

From \eqref{E:L-Four-def} we have $\Lop^*_\mu=\sum\limits_{n,m\ge0} \left(\Lop_{n,m}\right)^*_{\mu}$. From \eqref{E:kaptau} we see that $\left(\Lop_{n,m}\right)^*_{\mu}$ maps 
$L^2_{n,m}(\bndry D,\mu)$ to $L^2_{n,m}(\bndry D,\mu)$ via the formula
\begin{equation}\label{E:kaptau*}
\left(\Lop_{n,m}\right)^*_{\mu}(f)=\langle f,\tau_{n,m}\rangle\,\kappa_{n,m}.
\end{equation}
Of course, the norms of $\left(\Lop_{n,m}\right)^*_{\mu}$ and $\Lop^*$ match those of $\Lop_{n,m}$ and $\Lop$.

\begin{Prop}\label{P:l*l} The self-adjoint operator $\left(\Lop_{n,m}\right)^*_{\mu} \Lop_{n,m}$ has rank one with
spectrum given by 
\begin{equation*}
\{0, \|\Lop_{n,m}\|^2_\mu\}.
\end{equation*}
\end{Prop}

\begin{proof}
From \eqref{E:kaptau}, \eqref{E:kaptau*} and \eqref{E:kt1} we have
\begin{equation*}
\left(\Lop_{n,m}\right)^*_{\mu} \Lop_{n,m} f
=\|\tau_{n,m}\|_\mu^2 \langle f, \kappa_{n,m}\rangle \kappa_{n,m}
\end{equation*}
which is $\|\tau_{n,m}\|_\mu^2 \|\kappa_{n,m}\|_\mu^2=\|\Lop_{n,m}\|_\mu^2$ times the orthogonal projection onto the line through $\kappa_{n,m}$.  
\end{proof}

\begin{Remark}
It is clear that $\left(\Lop_{n,m}\right)^*_{\mu} \Lop_{n,m}$ admits an orthonormal basis of eigenfunctions.
\end{Remark}

\begin{Cor} If $\mu$ is admissible, then 
$\left(\Lop_{n,m}\right)^*_{\mu} \Lop_{n,m}$ is unitarily equivalent to $ \Lop_{n,m}\left(\Lop_{n,m}\right)^*_{\mu}$ in $L^2(bD, \mu)$. Moreover, if \  $\Lop$ is bounded in $L^2(bD, \mu)$, then $\left(\Lop\right)^*_{\mu} \Lop$ is unitarily equivalent to $ \Lop\left(\Lop\right)^*_{\mu}$.
\end{Cor}

Turning the attention to  essential norms and spectra, we 
recall that the essential norm of an operator $T$ on a Hilbert space $\mathcal H$ is the distance (in the operator norm) of $T$ from the space of compact operators $\mathcal K(\mathcal H)$ (see \cite{Pel}, page 25) while the essential spectrum
 of a bounded operator $T\in \mathcal L(H)$  is the spectrum of the projection of $T$ on the Calkin algebra $\mathcal L(H)/\mathcal K(H)$ (see \cite{HR}, page 32), and we have 
 
\begin{Prop}[\cite{HR}, Proposition 2.2.2]\label{R:ess-spec}

For a self-adjoint or anti-self-adjoint operator admitting an orthonormal basis of eigenfunctions,  the essential spectrum consists of limits of sequences of eigenvalues together with isolated eigenvalues of infinite multiplicity.
\end{Prop}
 In general, the essential spectrum includes the continuous spectrum, which is absent in our work but does appear in analysis of the Kerzman-Stein operator for many non-smooth planar domains (see \cite{Bol1}).
\begin{Prop}[see \cite{CM}, \S 3.2; {[{\"O}T]}, \S 2]\label{R:ess-norm}
 The essential norm of an operator $T$ is the 
square root of the largest value in the essential spectrum of $T^* T$. 
\end{Prop}

Using Propositions \ref {P:l*l}, \ref{R:ess-spec} and \ref{R:ess-norm} we obtain the following.

\begin{Thm}\label{T:Less1}
The essential norm of $\Lop$  on $L^2(\bndry D,\mu)$ is given by\begin{equation*}
\limsup\, \{\|\Lop_{n,m}\|_\mu\st {n,m\ge 0} \},
\end{equation*}
where  $\limsup q_{n,m}$ is defined by
\begin{equation*} 
\inf\,
\bigg\{
\sup\,
\big\{
q_{n,m}\in(\N\times\N)\setminus F
\big\}
 \st 
 F^{\text{finite}}\subset \N\times\N
\bigg\}.
\end{equation*}
The essential spectrum of $\Lop_\mu^*\Lop$ consists of $0$ together with all values of 
\begin{equation*}
\lim\limits_{j\to\infty} \|\Lop_{n_j,m_j}\|^2_\mu
\end{equation*}
taken along sequences $(n_j,m_j)$ with $\max\{n_j,m_j\}
\to\infty$ as $j\to \infty$ along which the above limit  exists.
\end{Thm}

As in the introduction we set \[\Aop_\mu=\Lop^*_\mu-\Lop\] and \[\left(\Aop_{n,m}\right)_\mu=\left(\Lop_{n,m}\right)^*_\mu-\Lop_{n,m}.\]
These operators are anti-self-adjoint.

\begin{Prop}\label{P:norm-Anm} 
$\left(\Aop_{n,m}\right)_\mu$ is a rank-two operator with norm given by
\[\|\left(\Aop_{n,m}\right)_\mu\|_\mu^2=
\|\Lop_{n,m}\|_\mu^2-1.\]
The spectrum  of $\left(\Aop_{n,m}\right)_\mu$ is the set
\begin{equation*}
\{0,\pm i\sqrt{\|\Lop_{n,m}\|_\mu^2-1}\}.
\end{equation*}
\end{Prop}

\begin{proof}
Using the notation from \eqref {E:kappa-tau}
 and the identities \eqref{E:kaptau}, \eqref{E:kt1} (but dropping subscripts),
set 
\begin{equation*}\label{E:lambda}
\lam=\tau-\dfrac{\langle\tau,\kappa\rangle}{\|\kappa\|_\mu^2}\kappa
=\tau-\dfrac{\kappa}{\|\kappa\|_\mu^2}.
\end{equation*}
 Then $\lam\perp\kappa$ and $\|\lam\|_\mu^2=\|\tau\|_\mu^2-\dfrac{|\langle\kappa,\tau\rangle|^2}{\|\kappa\|_\mu^2}
 =\|\tau\|_\mu^2-\dfrac{1}{\|\kappa\|_\mu^2}$.

Using \eqref{E:kaptau} and \eqref{E:kaptau*} we have
\begin{align}\label{E:A-kap-tau}
\left(\Aop_{n,m}\right)_\mu (f)
&=
\langle f,\tau\rangle\kappa - \langle f,\kappa\rangle\tau\notag\\
&=
\langle f,\lam\rangle\kappa - \langle f,\kappa\rangle\lam
\end{align}
so 
\begin{align}\label{E: Anm-nrm}
\|\left(\Aop_{n,m}\right)_\mu (f)\|_\mu^2
&= |\langle f,\lam\rangle|^2\|\kappa\|_\mu^2
+ |\langle f,\kappa\rangle|^2\|\lam\|_\mu^2.
\end{align}

If $\lambda =0$ then it follows $\Aop_{n,m} =0$ and also $\tau = \kappa /||\kappa||^2_{\mu}$, so \eqref{E:lkt} shows $||\Lop||=1$, which proves the desired result.

If on the other hand $\lambda\neq 0$ then we may write
 \eqref{E: Anm-nrm} as
\begin{align*}
\|\left(\Aop_{n,m}\right)_\mu (f)\|_\mu^2
&=
\left(
\left|\left\langle f,\frac{\lam}{\|\lam\|_\mu} \right\rangle\right|^2
+
\left|\left\langle f,\frac{\kappa}{\|\kappa\|_\mu} \right\rangle\right|^2
\right)
\|\kappa\|_\mu^2\|\lam\|_\mu^2.
\end{align*}

By Bessel's inequality this is less than or equal to
\begin{equation*}
\|f\|^2\|\kappa\|_\mu^2\|\lam\|_\mu^2
= \|f\|^2\left(\|\kappa\|_\mu^2\|\tau\|_\mu^2-1
\right)
\end{equation*}
with equality holding if and only if $f$ is in $\Span\{\kappa,\tau\}$.  
Thus 
\begin{equation}\label{E:anmkt}
\|\left(\Aop_{n,m}\right)_\mu\|^2=\|\kappa\|_\mu^2\|\tau\|_\mu^2-1.
\end{equation}

 The eigenvalues of $\Aopnm$ on $\Span\{\kappa,\tau\}$ are $\pm i\sqrt{\|\kappa\|_\mu^2\|\tau\|_\mu^2-1}$, and $\Aopnm$ vanishes on $\Span\{\kappa,\tau\}^\perp$.  Thus the spectrum of $\Aopnm$ is 
 \begin{equation*}
\{0,\pm i\sqrt{\|\kappa\|_\mu^2\|\tau\|_\mu^2-1}\}.
\end{equation*}
Invoking \eqref{E:lkt}, the proof is complete.
\end{proof}

\begin{Remark}
It is clear that $\left(\Aop_{n,m}\right)_\mu$ admits an orthonormal basis of eigenfunctions.
\end{Remark}

Assembling the pieces as in Theorems \ref{T:Lnrm} and \ref{T:Less1} 
we have the following.

\begin{Thm}\label{T:Anrm}
The norm  of $\Aop_\mu$ acting on $L^2(\bndry D, \mu)$ is
\begin{equation*}
\sup\left\{ \sqrt{\|\Lop_{n,m}\|_\mu^2-1}\st n,m\ge 0\right\}. 
\end{equation*}
The essential norm of $\Aop_\mu$ is given by
\begin{equation*}
\limsup\left\{ \sqrt{\|\Lop_{n,m}\|_\mu^2-1}\st n,m\ge0\right\}. 
\end{equation*}
In particular,  $\Aop_\mu$ is compact on $L^2(\bndry D, \mu)$
if and only if 
\begin{equation*}
\limsup\left\{ \sqrt{\|\Lop_{n,m}\|_\mu^2-1}\st n,m\ge0\right\}=0. 
\end{equation*}
The spectrum of $\Aop_\mu$ is the closure of
\begin{equation*}
\{0\}\cup\left\{\pm i\sqrt{\|\Lop_{n,m}\|_\mu^2-1}\st n,m\ge0\right\}.
\end{equation*}
The essential spectrum of $\Aop_\mu$ consists of $0$ together with all values of 
\begin{equation*}
\lim\limits_{j\to\infty} \pm i\sqrt{\|\Lop_{n_j,m_j}\|^2_\mu-1}
\end{equation*}
taken along sequences $(n_j,m_j)$ with $\max\{n_j,m_j\}
\to\infty$ as $j\to \infty$ along which the above limit  exists.
\end{Thm}

For a geometric interpretation of these results, let $\theta_{n,m}\in[0,\frac{\pi}{2})$ be the angle between $\kappa_{n,m}$ and $\tau_{n,m}$ in $L^2(\bndry D,\mu)$. (Thus $\langle \kappa_{n,m}, \tau_{n,m}\rangle=\cos\theta_{n,m}\cdot\|\kappa_{n,m}\|\cdot\|\tau_{n,m}\|$.) From \eqref{E:kt1}, \eqref{E:lkt} and \eqref{E:anmkt} we find that $\|\Lop_{n,m}\|_\mu=\sec\theta_{n,m}$ and $\|\left(\Aop_{n,m}\right)_\mu\|=\tan\theta_{n,m}$.

Returning to Proposition \ref{P:Hardy}, note that $\Lop$ will be the {\em orthogonal} projection from $L^2(\bndry D,\mu)$ to $H^2(\bndry D,\mu)$ (the {\em Szeg\H o projection} for $\mu$) if and only if $\Lop^*_\mu=\Lop$; this is in turn equivalent to any one of the following conditions:
\begin{itemize}
\item $\Aop_\mu=0$;
\item each $\left(\Aop_{n,m}\right)_\mu=0$;
\item each $\|\Lop_{n,m}\|_\mu=1$;
\item each $\theta_{n,m}=0$;
\item $\|\Lop\|_\mu=1$.
\end{itemize}
Examining \eqref{E:kappa-tau} we see that this will happen if and only if \[\omega(s)=c_{n,m} \left(\dfrac{  s}{r_1^2}\right)^{n}\left(\dfrac{1- s}{ r_2^2}\right)^{m}\] for each $(n,m)$ with  $c_{n,m}$ a positive constant. Selecting $(n,m)=(0,0), (1,0)$ and $(0,1)$ in turn we see that  this can happen if and only if $\frac{  s}{r_1^2}, \frac{1- s}{ r_2^2}$ and $\omega(s)$ are constant.  Applying \eqref{E:ds/dt}, \eqref{E:ds/dt-rev} and Proposition \ref{P:pcon} we obtain the following.

\begin{Prop}Let $D\in \tilde \RR$ and let $\mu$ be an admissible measure on $\bndry D$. Then
$\Lop$ will be the Szeg\H o projection for  $\mu$ if and only if $D$ is of the form $\{(z_1,z_2)\st a_1|z_1|^2 + a_2|z_2|^2<1\}$  and $\omega(s)$ is constant.
\end{Prop}

Bolt has shown that the Leray transform for a strongly ($\C$-linearly) convex bounded domain in $\C^n$ with $C^3$-smooth boundary will coincide with the Szeg\H o projection (for a suitably-chosen measure) if and only if the domain is a complex-affine image of the unit ball  (\cite{Bol2}, \cite{Bol3}).

\section{More on boundary measures and geometry}\label{S:bdm}

From formulas
\eqref{E:mono-bndry} and \eqref{E:piece-norm} 
we see that our theory becomes simplest with the use of the measure 
 \[d\mu_0 = \frac{1}{4\pi^2} \,ds\,d\theta_1\,d\theta_2\]
 on $\bndry D$.
When $D$ is smooth and strongly convex this measure will be comparable to surface measure $d\sigma$ but in general this will not be so.  Indeed, from \eqref{E:phi-grad},\eqref{E:ds/dt} and \eqref{E:ds/dt-rev} we find that
\begin{align}\label{E:surfRR}
d\sigma &= r_1 r_2\left(1+(\phi'(r_1))^2\right)^{1/2}dr_1\,d\theta_1\,d\theta_2\\
& = 
r_1 r_2\sqrt{dr_1^2+dr_2^2} \,d\theta_1 \,d\theta_2\notag\\
&=\frac{r_1^2 r_2^2}{ps(1-s)}
\sqrt{\left(\dfrac{s}{r_1}\right)^2+\left(\dfrac{1-s}{r_2}\right)^2}
\,ds\,d\theta_1\,d\theta_2,
\notag
\end{align}
where $p$ is as in \eqref{E:p-one}.
From \eqref{E:levinorm} and \eqref{E:leray-num} we deduce
\begin{equation}\label{E:levi-surf}
|\levi|\,d\sigma= \frac{1}{4\left(\left(\dfrac{s}{r_1}\right)^2+\left(\dfrac{1-s}{r_2}\right)^2\right)}\,ds\,d\theta_1\,d\theta_2.
\end{equation}

 Lemma \ref{L:squ}  now shows that $d\mu_0$ is comparable to $|\levi|\,d\sigma$.
  In particular we see that $d\mu_0$
   will not be comparable to $d\sigma$ unless $|\levi|$ is bounded above and below.  For $D\in\PP$, for example, it is easy to check using \eqref{E:gamma-PP} that this happens if and only if $p=2$.
   
Formula \eqref{E:levi-surf} motivates the following

\begin{Def}\label{D:ord-q}
We will say that a rotation-invariant measure on the boundary of a domain $D\in\tilde\RR$ has {\em order $q$} if it is a continuous positive multiple of $|\levi|^{1-q}\,d\sigma$.
\end{Def}

Thus surface measure $d\sigma$ has order $q=1$, and the special measure $d\mu_0$
has order $q=0$.

 From \eqref{E:surfRR} and \eqref{E:levi-surf} we find 
\begin{equation*}\label{E:Levi-size}
|\levi| 
=
\frac{ps(1-s)}
{4r_1^2r_2^2
\left(\left(\dfrac{s}{r_1}\right)^2+\left(\dfrac{1-s}{r_2}\right)^2\right)^{3/2}}.
\end{equation*}
It follows that $\mu$ has order $q$ if and only if $\mu$ is expressed as
\begin{equation}\label{E:ord-q}
\varphi(s)\left(\dfrac{r_1^2r_2^2}{ps(1-s)}\right)^q\,ds\,d\theta_1\,d\theta_2,
\end{equation}
where $\varphi$ is assumed to be positive and continuous on all of $\bndry D$.
In particular, the {\em Fefferman measure} $d\mu^{\sf{Fef}}_{\bndry D}$ has order $q=2/3$, where
\begin{equation*}
d\mu^{\sf{Fef}}_{\bndry D}\eqdef|\levi|^{1/3}\,d\sigma
=
\left(\dfrac{r_1^2r_2^2}{2ps(1-s)}\right)^{2/3}\,ds\,d\theta_1\,d\theta_2
\end{equation*}
(see p. 259 of \cite{Fef}; also \cite{Bar1}).  This measure may be defined on general smooth pseudoconvex domains in $\C^2$ and plays a distinguished role in complex analysis due to the fact that it transforms
by the rule
\[
F^*\left(d\mu^{\sf{Fef}}_{\bndry D_2}\right) =|\det F'|^{2/3} d\mu^{\sf{Fef}}_{\bndry D_1}
\]
 under a biholomorphic mapping $F$ mapping $\bndry D_1$ to $\bndry D_2$.  (Modified versions of this construction work also in higher dimensions.)

Note for comparison that integrals of the form $\int_{\bndry D} |\levi|^{-1}\,d\sigma$ (corresponding to $q=2$) appear in work on spectral asymptotics of the $\bar\partial$-Neumann problem by Metivier \cite{Met}.

 If $D\in\RR$ then combining Proposition \ref{P:RRdom} with  \eqref{s-asymp}, \eqref{(1-s)-asymp} and \eqref{E:ord-q} we find that a measure of order $q$ is given by the following expression
 \begin{multline}\label{E:ord-q-rr}
 \varphi (s)\left(s^{\frac{2}{p_1}-1}(1-s)^{\frac{2}{p_2}-1}\right)^q\,ds\,d\theta_1\,d\theta_2 \\
 = \varphi (s)\,s^{q\left(\frac{1}{p_1}-\frac{1}{p_1^*}\right)}(1-s)^{q\left(\frac{1}{p_2}-\frac{1}{p_2^*}\right)}\,ds\,d\theta_1\,d\theta_2,\end{multline}
where $\varphi$ is positive and continuous on $bD$.
Recalling Definition \ref{D:adm} we easily obtain the following result.

\begin{Prop} \label{P:padm}
If $D\in\RR$ then a rotation-invariant measure of order $q$ is admissible if and only if \eqref{E:rc0} holds.
\end{Prop}

Applying this to values of $q$ just discussed we see that $q=2/3$ (indeed, any $q\in[0,1]$) will always work, while $q=2$ works if and only if both of the $p_j$ lie in the interval $(\frac43, 4)$.

For later reference, we close this section with  a quick look at the differential geometry of $\bndry D$.  A computation based on the parameterization \eqref{E:s-theta-param} shows that the principal curvatures of $\bndry D$ are given by 
\begin{align}
 \kappa_1&=\frac{s}{r_1^2\sqrt{\left(\frac{s}{r_1}\right)^2+\left(\frac{1-s}{r_2}\right)^2}}\notag\\
 \kappa_2&=\frac{1-s}{r_2^2\sqrt{\left(\frac{s}{r_1}\right)^2+\left(\frac{1-s}{r_2}\right)^2}}\label{E:pr-curv}\\
\kappa_3&=(p-1)\frac{s(1-s)}{r_1^2r_2^2}\frac{1}{\left(\left(\frac{s}{r_1}\right)^2+\left(\frac{1-s}{r_2}\right)^2\right)^{3/2}}.\notag
\end{align}
For a point in $\gamma_+$ the corresponding principal directions are given by $(i,0)$, $(0,i)$ and the tangent to $\gamma_+$.  The principal directions at other points are found by rotation.
\bigskip

As a check, note that for the unit sphere we have  $p=2$, $r_1=\sqrt{s}$, $r_2=\sqrt{1-s}$  and thus $\kappa_1=\kappa_2=\kappa_3=1$.

\section{Asymptotics in $\RR$}\label{S:as}

In this section we perform asymptotic analysis of the norms of 
$\Lop_{n,m}$ and use these results to prove  Theorems \ref {T:rress} and \ref {T:rress-ks}.

\begin{Thm}\label{T:asymp} 
Suppose that $D\in\RR$ and that $\mu$ is an admissible measure on $\bndry D$ of order $q$ (as  in Definition \ref{D:ord-q}).

Let $(n_j, m_j)$ be a sequence in $\N\times\N$ with $\max\{n_j,m_j\}\to\infty$ as $j\to\infty$. 
\begin{enumerate}[{\normalfont (a)}]
\item If   $\min\{n_j,m_j\}\to\infty$ and $\frac{n_j}{m_j}\to u\in[0,\infty]$ then 
\begin{equation*}\label{E:case1}
\|\Lop_{n_j,m_j}\|_\mu^2\to \frac{\sqrt{\breve p\left(\frac{u}{1+u}\right) \breve p^*\left(\frac{u}{1+u}\right)}}{2},
\end{equation*}
where $\breve p$ was defined in \eqref{E:pb-def}. \label{I:u}
\item If $n_j$ is independent of $j$ then \[\|\Lop_{n_j,m_j}\|_\mu^2\to \frac{\Gamma\left(\frac{2n_0}{p_1}+1-q\left(\frac{1}{p_1}-\frac{1}{p_1^{*}}\right)\right)
\Gamma\left(\frac{2n_0}{p_1^{*} }+1-q\left(\frac{1}{p_1^{*}}-\frac{1}{p_1}\right)\right)
}
{\Gamma^{2}(n_0+1)
\left(\frac{2}{p_1}\right)^{\frac{2n_0}{p_1 }+1-q\left(\frac{1}{p_1}-\frac{1}{p_1^{*}}\right)}
\left(\frac{2}{p_1^{*}}\right)^{\frac{2n_0}{p_1^{*}}+1-q\left(\frac{1}{p_1^{*}}-\frac{1}{p_1}\right)}
}
\]
where $n_0$ is the common value of the $n_j$ and $p_1$ is as in Definition \ref{D:rr}.  \label{I:n-const}
\item If  $m_j$ is independent of $j$ then \[\|\Lop_{n_j,m_j}\|_\mu^2\to 
\frac{\Gamma\left(\frac{2m_0}{p_2}+1-q\left(\frac{1}{p_2}-\frac{1}{p_2^{*}}\right)\right)
\Gamma\left(\frac{2m_0}{p_2^{*} }+1-q\left(\frac{1}{p_2^{*}}-\frac{1}{p_2}\right)\right)
}
{\Gamma^{2}(m_0+1)
\left(\frac{2}{p_2}\right)^{\frac{2m_0}{p_2 }+1-q\left(\frac{1}{p_2}-\frac{1}{p_2^{*}}\right)}
\left(\frac{2}{p_2^{*}}\right)^{\frac{2m_0}{p_2^{*}}+1-q\left(\frac{1}{p_2^{*}}-\frac{1}{p_2}\right)}
}
\]
where $m_0$ is the common value of the $m_j$ and $p_2$ is as in Definition \ref{D:rr}\label{I:m-const}.
\end{enumerate}
\end{Thm}

\begin{proof} [Proof of Theorems \ref{T:rress} and \ref{T:rress-ks}, assuming Theorem \ref{T:asymp}.]
Suppose that 
\begin{enumerate}
\item[(*)]
{\em $ \lim\|\Lop_{n_j,m_j}\|_\mu$
  exists (possibly equal to $+\infty$) along some 
sequence $(n_j,m_j)$ with $\mathrm{max}\{n_j,m_j\}
\to\infty$ as $j\to \infty$.}
\end{enumerate}
 We consider the sequence of quotients: $\{n_j/m_j\} \subset [0,+\infty]$
 and distinguish the two cases: $\mathrm{min}\{n_j, m_j\} \to\infty$;  
$\mathrm{min}\{n_j, m_j\} \not\to\infty$. In either case, passing to a subsequence 
 we may arrange for one of 
 the following three conditions to hold:
\refstepcounter{equation}\label{N:trichot}
\begin{enum}
\item $\min\{n_j,m_j\}\to\infty$ and $\frac{n_j}{m_j}$ approaches some value $u\in[0,\infty]$, or
\item $n_j$ is independent of $j$, or
\item $m_j$ is independent of $j$.
\end{enum}
Then Theorem \ref{T:asymp} provides the limiting value of 
$\|\Lop_{n_j,m_j}\|_\mu$ in each of these three cases.  

In particular, since  Theorem \ref{T:asymp}  shows that none of these limiting values can be infinite, we conclude that the set $\{\|\Lop_{n,m} \|_\mu\st {n,m\ge 0} \}$ is bounded.  Then Theorem \ref{T:Lnrm} shows that 
$\Lop$ is bounded on $L^{2}(\bndry D,\mu)$. 

The orthonormal basis of eigenfunctions for $\Lop_\mu^*\Lop$ is obtained by combining eigenfunction bases for each $L^2_{n,m}(\bndry D,\mu)$.

Finally, we use Theorem \ref{T:Less1} and the above description of limiting values of 
$\|\Lop_{n_j,m_j}\|_\mu$ to verify the description of the essential spectrum of $\Lop_\mu^*\Lop$.  

This completes the proof of Theorem \ref{T:rress}.  The proof of Theorem \ref{T:rress-ks} proceeds in similar fashion using  Theorem \ref{T:Anrm}.
\end{proof}

\begin{proof} [Proof of Theorem \ref{T:asymp}, part (\ref{I:u}).] 
This is a variation on Laplace's method for asymptotic expansion of integrals. (See for example Chapter 3 of \cite{Mil}.)

From \eqref{E:piece-norm} and \eqref{E:beta} we have
\begin{equation}\label{E:iks}
\|\Lop_{n,m}\|_\mu^2 = \frac{I_{n,m,-1}\cdot I_{n,m,1}}{I_{n,m,0}^2}, 
\end{equation}
where 
\begin{align*}
I_{n,m,k} &= \int_0^1 g_{1,k}^n g_{2,k}^m h_k \,ds,\\
g_{1,k} &= r_1^{2k} s^{1-k}, \\
g_{2,k} &= r_2^{2k} (1-s)^{1-k}, \\
h_k &= \omega^k.
\end{align*}

Using \eqref{E:ds/dt} and \eqref{E:ds/dt-rev} we have 
\begin{align}\label{E:nm-diff}
\frac{d}{ds}\left(\log g_{1,k}^n g_{2,k}^m \right)
&= 
\left( \frac{2k}{\breve p(s)} + 1 - k\right) \frac{n-(n+m)s}{s(1-s)}\notag\\\
&=\left( \frac{2k}{\breve p(s)} + 1 - k\right)
\left(\frac{n}{s}-\frac{m}{1-s}\right).
\end{align}
Noting that
\begin{equation*}
\frac{2k}{\breve p(s)} + 1 - k
= \begin{cases}
2/\breve p^*(s) & \text{ for } k=-1\\
1 & \text{ for } k=0\\
2/\breve p(s) & \text{ for } k=1\\
\end{cases}
\end{equation*}
and recalling \eqref{E:p-ext} we see that 
\begin{align}
\frac{2k}{\breve p(s)} + 1 - k&\le 2,\label{E:uqbd}\\
\label{E:lqbd}
C_k\eqdef \inf\limits_{0\le s \le 1}\left(\frac{2k}{\breve p(s)} + 1 - k\right)&>0.
\end{align}

It follows easily that $\log g_{1,k}^n g_{2,k}^m$ takes its maximum value at \[s_{n,m}\eqdef \frac{n}{n+m}.\]  (We'll assume for the remainder of this proof that $n,m>0$ and thus $0<s_{n,m}<1$.)
  Integrating \eqref{E:nm-diff} from $s_{n,m}$ to $s$ and applying \eqref{E:lqbd} we find that in fact 
\begin{equation} \label{E:dom1}
\log 
\left(\frac{g_{1,k}^n(s) g_{2,k}^m(s)}{g_{1,k}^n(s_{n,m})g_{2,k}^m(s_{n,m})}\right)\le 
C_k
\log 
\left(\frac{s^n (1-s)^m}{s_{n,m}^n(1-s_{n,m})^m}\right).
\end{equation}

We set
\begin{equation}\label{E:amk}
A_{n,m,k} = \sqrt{ \frac{2nm}{\left( \frac{2k}{\breve p(s_{n,m})} + 1 - k\right) (n+m)^3}}.
\end{equation}
(The reader interested in tracing the motivation for the computations to come may wish to note that $A_{n,m,k}=\sqrt{-2/\left( \log g_{1,k}^n g_{2,k}^m\right)''(s_{n,m})}$, though $\left(\log  g_{1,k}^n g_{2,k}^m\right)''(s)$ may not exist for other values of $s$.)

Using \eqref{E:uqbd} and \eqref{E:lqbd} it is easy to check that 
\begin{align}
A_{n,m,k} & \ge  \sqrt{\frac{nm}{(n+m)^3}}\label{E:abdd1}\\
&=s_{n,m}\sqrt{\frac{m}{n(n+m)}}\notag\\
&=(1-s_{n,m})\sqrt{\frac{n}{m(n+m)}},\label{E:a-low}\notag\\
A_{n,m,k} &\le \frac{\sqrt{2} s_{n,m}}{\sqrt{C_k n}},\\
A_{n,m,k} &\le \frac{\sqrt{2} (1-s_{n,m})}{\sqrt{C_k m}}\label{E:abdd2},\\
A^2_{n,m,k} &\le \frac{2}{C_k} s_{n,m}^2 (1-s_{n,m})\label{E:quad1},\\
A^2_{n,m,k} &\le \frac{2}{C_k} s_{n,m} (1-s_{n,m})^2\label{E:quad2},
\end{align}
where $C_k$ is as in \eqref{E:lqbd}.
In particular we also have 
\begin{equation}\label{E:amko}
A_{n,m,k} = o( s_{n,m} ) \text{ and } A_{n,m,k} = o(1- s_{n,m} )
\text{ as } \min\{n,m\} \to \infty.
\end{equation}

We define  functions $g_{n,m,k}$ and  $h_{n,m,k}$ on $\R$ by setting
\begin{align}
g_{n,m,k}(t) &= \dfrac{g_{1,k}^n(  s_{n,m} + t A_{n,m,k} ) g_{2,k}^m( s_{n,m} + t A_{n,m,k} )}
{g_{1,k}^{n}(s_{n,m}) g_{2,k}^{m}(s_{n,m})}\label{E:gnmk-def}\\
h_{n,m,k}(t)&= \dfrac{h_k(  s_{n,m} + t A_{n,m,k} ) }{h_k(s_{n,m}) }
\label{E:hnmk-def}
\end{align} 
for $t\in J_{n,m,k}\eqdef \left( -\frac{s_{n,m}}{A_{n,m,k}}, \frac{1-s_{n,m}}{A_{n,m,k}}\right)$,  with
$
g_{n,m,k}(t) =h_{n,m,k}(t)= 0
$
otherwise. 

Note  that $g_{n,m,k}(t)$  assumes a maximum value of $1$ at $t=0$. Note also that from  \eqref{E:dom1} and \eqref{E:uqbd} and the monotonicity properties of $g_{n,m,0}$ we have
\begin{align}\label{E:gnmk-est1}
g_{n,m,k}(t) \le g_{n,m,0}^{C_k}\left(\frac{t A_{n,m,k}}{A_{n,m,0}}
 \right)
\le g_{n,m,0}^{C_k}\left(\frac{t}{\sqrt{2}}
 \right)
\end{align}
for $k=\pm 1$.

We claim that
\begin{align}\label{E:gnmk-est2}
g_{n,m,k}(t)  \le e^{C_k(1-2^{-|k|/2}|t|)}\text{ for all $t$ when $\min\{n,m\}\ge 2, \, k=-1,0,1$.}
\end{align}
In view of \eqref{E:gnmk-est1} it will suffice to prove
\begin{align}\label{E:gnmk-est0}
g_{n,m,0}(t)  \le e^{1-|t|}\text{ for all $t$ when $\min\{n,m\}\ge 2$.}
\end{align}
This is trivial for $t\notin J_{n,m,0}$.  For $|t|<1$ it follows from $g_{n,m,k}(t)\le 1$.  
For  $t\in J_{n,m,0}$, $|t|>1$, we use \[g_{n,m,0}(t)
=\left(1+t\sqrt{\frac{2m}{n(n+m)}}\right)^n\!
\left(1-t\sqrt{\frac{2n}{m(n+m)}}\right)^m\] to verify that
$(\log g_{n,m,0})''(t)<0$ so that $\log g_{n,m,0}$ is concave on $ J_{n,m,k}$.  In particular,  for $t\in J_{n,m,0}, t>1$ we have
\begin{align*}
\log g_{n,m,0} (t)
& \le \log g_{n,m,0} (1) + (\log g_{n,m,0})'(1) \cdot(t-1)\\
&\le 0+(\log g_{n,m,0})'(1)\cdot (t-1)\\
&= -\frac{2}{\left(1+\sqrt{\frac{2m}{n(n+m)}}\right)\left(1-\sqrt{\frac{2n}{m(n+m)}}\right)}\cdot(t-1)\\
&\le -\frac{2}{1+\sqrt{\frac{2m}{n(n+m)}}}\cdot(t-1)\\
&\le -\frac{2}{1+\sqrt{\frac{2}{n}}}\cdot(t-1)\\
& \le -(t-1)
\end{align*}
showing that \eqref{E:gnmk-est0} holds. A symmetric argument shows that \eqref{E:gnmk-est0} holds also in the remaining case: 
$t\in J_{n,m,0}, t<-1$.

Next we consider the pointwise behavior of $g_{n,m,k}(t)$ as $\min\{n,m\}\to\infty$.  Note that \eqref{E:amko} guarantees that  each fixed $t$ lies in $J_{n,m,k}$ when $\min\{n,m\}$ is large enough.  
Using \eqref{E:nm-diff} we have
\begin{multline*}
(\log g_{n,m,k})'(t)
=
-
\left( \frac{2k}{\breve p(s_{n,m}+A_{n,m,k} t)} + 1 - k\right)\\
\cdot\frac{(n+m) 
A_{n,m,k}^2 t}{(s_{n,m}+A_{n,m,k} t)(1-s_{n,m}-A_{n,m,k}t)}
\end{multline*}
and so
\begin{multline*}
\log g_{n,m,k}(t)
=- (n+m) 
A_{n,m,k}^2\int_{0}^t \left( \frac{2k}{\breve p(s_{n,m}+A_{n,m,k} \tau)} + 1 - k\right)\\
\cdot \frac{ \tau}{(s_{n,m}+A_{n,m,k} \tau)(1-s_{n,m}-A_{n,m,k}\tau)}\,d\tau.
\end{multline*}
Letting $\min\{n,m\}\to\infty$ we find with the use of \eqref{E:abdd1}, \eqref{E:abdd2}  that the above integral is asymptotic to
\begin{multline*}
\int_{0}^t \left( \frac{2k}{\breve p(s_{n,m} )} + 1 - k\right)
\cdot \frac{ \tau}{s_{n,m}(1-s_{n,m})}\,d\tau\\
=
\left( \frac{2k}{\breve p(s_{n,m})} + 1 - k\right) \frac{t^2}{2 s_{n,m} (1-s_{n,m})};
\end{multline*}
 hence
\begin{align}\label{E:gnmk-lim}
\lim & \,g_{n,m,k}(t)\notag\\
&= \lim \exp\left( -(n+m) A_{n,m,k}^2 \left( \frac{2k}{\breve p(s_{n,m})} + 1 - k\right) \frac{t^2}{2 s_{n,m} (1-s_{n,m})}
\right)\notag\\
&= e^{-t^2}.
\end{align}

Turning now to $h_{n,m,k}$, see \eqref{E:hnmk-def}, we first use \eqref{E:ord-q-rr} and \eqref{E:rc0} to verify that $h_k$ takes the form 
\begin{equation}\label{E:hk-form}
h_k(s)=\phi(s) s^{B_1} (1-s)^{B_2}
\end{equation}
 with $\phi$ positive and continuous on $[0,1]$ and $B_1, B_2\in(-1,1)$.
Then we see
\begin{equation}\label{E:hmnk-bdd}
h_{n,m,k}(t) \le C \text{ when }
{-\frac{s_{n,m}}{2A_{n,m,k}}}\le t \le{\frac{1-s_{n,m}}{2A_{n,m,k}}}
\end{equation}
and using \eqref{E:hnmk-def} and \eqref{E:amko} we obtain
\begin{equation}\label{E:hmnk-lim}
h_{n,m,k}(t)\to 1 \text{ as $\min\{n,m\}\to\infty$, uniformly on bounded sets.}
\end{equation}

We claim that also 
\begin{equation}\label{E:hnmk-move}
\int_{-\infty}^\infty \left|h_{n,m,k}(t) - 1\right|
e^{C_k\left(1-2^{-|k|/2}|t|\right)}\,dt \to  0 \text{ as $\min\{n,m\}\to\infty$.}
\end{equation}
 To see this, decompose  
the integral into five pieces
\begin{equation*}\label{E:5cut}
\int_{-\infty}^{-\frac{s_{n,m}}{A_{n,m,k}}}
+
\int_{-\frac{s_{n,m}}{A_{n,m,k}}}^{-\frac{s_{n,m}}{2A_{n,m,k}}}
+
\int_{-\frac{s_{n,m}}{2A_{n,m,k}}}^{\frac{1-s_{n,m}}{2A_{n,m,k}}}
+
\int_{\frac{1-s_{n,m}}{2A_{n,m,k}}}^{\frac{1-s_{n,m}}{A_{n,m,k}}}
+
\int_{\frac{1-s_{n,m}}{A_{n,m,k}}}^\infty.
\end{equation*}
The first and fifth terms involve intervals on which $h_{n,m,k}$ vanishes, so they reduce to $\int_{-\infty}^{-\frac{s_{n,m}}{A_{n,m,k}}} e^{C_k\left(1+2^{-|k|/2}t\right)}\,dt$
and
$\int_{\frac{1-s_{n,m}}{A_{n,m,k}}}^\infty e^{C_k\left(1-2^{-|k|/2}t\right)}\,dt
$, respectively, and thus 
tend to zero by \eqref{E:amko}.  
Using the dominated convergence theorem with the support of \eqref{E:hmnk-bdd} and \eqref{E:hmnk-lim} we see that the third term tends to zero.  Setting $v=\frac{s_{n,m} + t A_{n,m,k} }{s_{n,m}}$ 
we find that the second term may be written as
\begin{equation*}
\frac{s_{n,m}}{A_{n,m,k}}
\int_0^{1/2}
\left|
\frac{\phi(v s_{n,m})}{\phi(s_{n,m})} v^{B_1}\left(\frac{1-vs_{n,m}}{1-s_{n,m}}\right)^{B_2}
-1
\right|
e^{C_k\left(1+\frac{s_{n,m}(v-1)}{2^{|k|/2}A_{n,m,k}}\right)}
\,dv;
\end{equation*}
 this is bounded by a constant times 
 \begin{align*}
\frac{s_{n,m}}{A_{n,m,k}(1-s_{n,m})} 
e^{-\frac{C_k s_{n,m}}{2^{1+|k|/2}A_{n,m,k}}}
& \le \frac{2}{C_k} \left(\frac{s_{n,m}}{A_{n,m,k}} \right)^3
e^{-\frac{C_k s_{n,m}}{2^{1+|k|/2}A_{n,m,k}}}
& \\
&\to 0.
\end{align*}
(The inequality stems from \eqref{E:quad1}.)
  A similar argument takes care of the fourth term.

We are now ready to compute that
\begin{align}
\frac{I_{n,m,k}}
{A_{n,m,k} h_k(s_{n,m}) g_{1,k}^n(s_{n,m}) g_{2,k}^m(s_{n,m})} &= 
\frac{\int_0^1 g_{1,k}^n(s) g_{2,k}^m(s) h_k(s)\,ds}
{A_{n,m,k} h_k(s_{n,m}) g_{1,k}^n(s_{n,m}) g_{2,k}^m(s_{n,m})}\notag\\
&=
\int_{-\infty}^\infty h_{n,m,k}(t) g_{n,m,k}(t) \, dt \notag\\
&=\int_{-\infty}^\infty \left(h_{n,m,k}(t)-1\right) g_{n,m,k}(t) \, dt\label{E:keylime}\\
&\,\qquad+
\int_{-\infty}^\infty  g_{n,m,k}(t) \, dt \notag
\\
&\to 
0+\int_{-\infty}^\infty e^{- t^2} \, dt =\sqrt{\pi}\notag
\end{align}
as $\min\{n,m\}\to\infty$, where we have used \eqref{E:gnmk-est2} and \eqref{E:hnmk-move} to find the limit of the first term and \eqref{E:gnmk-est2}, \eqref{E:gnmk-lim} and dominated convergence to find the limit of the second term.

Combining these results and simplifying we have
\begin{align}
\frac{2}{\sqrt{\breve p(s_{n,m}) \breve p^*(s_{n,m})}} \|\Lop_{n,m}\|_\mu^2
&= \frac{A_{n,m,0}^2}{A_{n,m,-1}A_{n,m,1}}\cdot\frac{I_{n,m,-1}\cdot I_{n,m,1}}{I_{n,m,0}^2}\notag\\
&\to 1 \label{E:keycomb}
\end{align}
as $\min\{n,m\}\to\infty$, which implies part (\ref{I:u}) of
Theorem \ref{T:asymp}. 
\end{proof}

\begin{Lem}\label{L:wl}
Suppose that
\begin{itemize}
\item $H$ is a continuous function on $[0,1]$;
\item $H(0)\ne0$;
\item $g$ is a non-negative $C^1$ function on $[0,1]$ with a strict maximum at $0$;
\item $g'(0)<0$;
\item $\sigma>-1$.
\end{itemize}
Then
\begin{equation}\label{E:wlv}
\int_0^1  g^m(s) s^\sigma H(s) \,ds \sim H(0) \left(\frac{-g'(0)}{g(0)}\right)^{-\sigma-1}  \Gamma(\sigma+1)   m^{-\sigma-1}
g^m(0)
\end{equation}
as $m\to\infty$.
\end{Lem}

\begin{proof}
This is a minor variation of Watson's lemma. (See for example Chapter 2 of \cite{Mil}.)

First note that the hypotheses on $g$ imply that 
$\frac{\log(g(s)/g(0))}{s}$ extends to a negative continuous function on $[0,1]$ and so
$g(s)\le g(0)\exp(-\eps s)$ for some $\eps>0$; thus for some $C>0$ we have
\begin{equation}\label{E:wl-bdd}
\left(\frac{g(s/m)}{g(0)}\right)^m s^\sigma H(s/m)\le C \exp(-\eps s) s^\sigma
\end{equation}
on $[0,m]$. Also note that
\begin{equation}\label{E:wl-lim}
m\left(\log g(s/m)-\log g(0)\right) \to s(\log g)'(0)=\frac{g'(0)}{g(0)} s
\end{equation}
as $m\to\infty$.

Using the change of variables formula and the dominated convergence theorem  with the support of \eqref{E:wl-bdd} and \eqref{E:wl-lim} we have
\begin{align*}
m^{\sigma+1} g^{-m}(0) &\int_0^1 g^m(s) s^\sigma H(s)\,ds\\
= &\int_0^m \left(\frac{g(s/m)}{g(0)}\right)^m s^\sigma H(s/m)\,ds\\
\to & \int_0^\infty \exp\left( \frac{g'(0)}{g(0)} s\right) s^\sigma H(0)\,ds\\
=& H(0) \left(-\frac{g(0)}{g'(0)}\right)^{\sigma+1} \int_0^\infty \exp(-s)s^\sigma\,ds\\
=& H(0) \left(-\frac{g(0)}{g'(0)}\right)^{\sigma+1}  \Gamma(\sigma+1)
\end{align*}
which is equivalent to \eqref{E:wlv}.
\end{proof}

\begin{proof} [Proof of Theorem \ref{T:asymp}, part (\ref{I:n-const}).]  We focus on the same three integrals as in the proof of part \eqref{I:u}.

Let $n_0$ be the common value of the $n_j$. 
To apply Lemma \ref{L:wl}, we use \eqref{E:r1-asympt} and \eqref{E:ord-q-rr}  to match the integrals to the left-hand side of \eqref{E:wlv} and we use \eqref{E:r2-s} to evaluate $\frac{d\log r_2}{ds}$ and $\frac{d\log ((1-s)/r_2)}{ds}$ at $s=0$.  The resulting approximations read as follows: 
\begin{align*}
\bullet
 \int_0^1 &\!s^{n_0}(1-s)^{m_j}\,ds
\sim \Gamma(n_0+1) m_j^{-n_0-1}\ ; \\
\bullet
 \int_0^{1} &\!r_1^{2n_0} r_2^{2m_j}\omega(s)\,ds \ \\
&\qquad\quad\ \sim \frac
{h(0)\, p_1^{1+q\left(\frac{1}{p_1}-\frac{1}{p_1^*}\right)} b_2^{2m_j+\frac{2n_0}{p_1}}\Gamma\left(\frac{2n_0}{p_1}+1+q\left(\frac{1}{p_1}-\frac{1}{p_1^*}\right)\right)}
{c_2^{\frac{2n_0}{p_1}} (2m_j)^{\frac{2n_0}{p_1}+1+q\left(\frac{1}{p_1}-\frac{1}{p_1^*}\right)}}\ ;\\
\bullet
 \int_0^{1}&\!\! \left(\dfrac{  s}{r_1}\right)^{\!2n_0}
 \!\! \left(\dfrac{1- s}{ r_2}\right)^{\!2m_j}\!\!
 \frac{1}{\omega(s)}\,ds\\
&\qquad\quad\ \sim \frac
{p_1^{\frac{2n_0}{p_1}} (p^*_1)^{\frac{2n_0}{p_1^*}+1-q\left(\frac{1}{p_1}-\frac{1}{p_1^*}\right)}c_2^{\frac{2n_0}{p_1}}\Gamma\left(\frac{2n_0}{p^*_1}+1-q\left(\frac{1}{p_1}-\frac{1}{p_1^*}\right)\right)}
{h(0)\, b_2^{2m_j+\frac{2n_0}{p_1}} (2m_j)^{\frac{2n_0}{p^*_1}+1-q\left(\frac{1}{p_1}-\frac{1}{p_1^*}\right)}}.
\end{align*}
  
  Plugging these results into \eqref{E:piece-norm} and \eqref{E:beta}
we find that 
\begin{equation}\label{E:n-const-lim}
\|\Lop_{n_0,m_j}\|_\mu^2
\to 
\frac{\Gamma\left(\frac{2n_0}{p_1}+1+q\left(\frac{1}{p_1}-\frac{1}{p_1^{*}}\right)\right)
\Gamma\left(\frac{2n_0}{p_1^{*} }+1+q\left(\frac{1}{p_1^{*}}-\frac{1}{p_1}\right)\right)
}
{\Gamma^{2}(n_0+1)
\left(\frac{2}{p_1}\right)^{\frac{2n_0}{p_1 }+1+q\left(\frac{1}{p_1}-\frac{1}{p_1^{*}}\right)}
\left(\frac{2}{p_1^{*}}\right)^{\frac{2n_0}{p_1^{*}}+1+q\left(\frac{1}{p_1^{*}}-\frac{1}{p_1}\right)}
}
\end{equation}
as claimed.
\end{proof}

\begin{proof} [Proof of Theorem \ref{T:asymp}, part (\ref{I:m-const}).]
This is parallel to the proof of part (\ref{I:n-const}).
\end{proof}

\section{Examples} \label{S:ce}

\begin{Ex}\label{E:x1}
Let $\breve p$ be a smooth map from $[0,1]$ to $[1,\infty)$ satisfying 
\begin{itemize}
\item $\breve p\left(\frac12\right)=1$;
\item $\breve p(s)>1$ for $s\ne \frac12$;
\item $\breve p(s) \equiv 2$ for $s$ near $0$ and for $s$ near $1$.  
\end{itemize}
Let $D$ be the domain generated by $\breve p, 1, 1$ as in 
\eqref{E:s-theta-param} and Definition \ref{D:gen}. 
We claim that $D$ has the following properties:
\refstepcounter{equation}\label{N:exprop1}
\begin{enum}
\item $D$ is a  convex Reinhardt domain with $C^\infty$-smooth  boundary;
\item $D\notin \tilde \RR$; \label{I:ntr1}
\item $D$ is strictly convex (i.e., $\bndry D$ contains no line segments);
\item the principal curvatures $\kappa_1$ and $\kappa_2$ are strictly positive, but $\kappa_3$ vanishes precisely on the  torus corresponding to $s=1/2$;
\item the conditions of Theorem \ref{T:adm} are still equivalent and thus can still be used to define the notion of an admissible measure as in Definition \ref{D:adm}; \label{I:adm1}
\item measures of order $q$  are admissible for all $q\in\R$; \label{I:qadm1}
\item $\Lop_D$ fails to be bounded on $L^2(\bndry D,\mu)$ 
when $\mu$ is any admissible measure  given by $\frac{1}{4\pi^2}\omega(s)\,ds\,d\theta_1\,d\theta_2$ with $\omega(s)$ positive and continuous for $s\in(0,\frac12)\cup(\frac12, 1)$.\label{I:unbdd1}
\end{enum}

The first four items follow easily from material in the beginning of \S \ref{S:geo} along with \eqref{E:pr-curv} (see in particular Theorem \ref{T:rr-tilde-cond} to check item \itemref{N:exprop1}{I:ntr1}).  Item \itemref{N:exprop1}{I:adm1} follows from the continued validity of the conclusions of Lemmas \ref{L:HB} and \ref{L:squ}. Item \itemref{N:exprop1}{I:qadm1} follows from \eqref{E:ord-q}.  (In fact, the value of $q$ is irrelevant here.)

To verify item \itemref{N:exprop1}{I:unbdd1}, consider a sequence  $(n_j, m_j)$  in $\N\times\N$ with $\frac{n_j}{m_j}\to u\in (0,1)\cup(1,\infty)$.
Referring to the proof of Theorem \ref{T:asymp}, part (\ref{I:u}) and in particular to \eqref{E:keylime} we find that $h_{n_j,m_j,k}(t) g_{n_j,m_j,k}(t) \to e^{-t^2}$ uniformly for $-1\le t\le 1$.  Truncating the integral we see that 
\begin{equation*}
\liminf\frac{I_{n_j,m_j,k}}
{A_{n_j,m_j,k} h_k(s_{n_j,m_j}) g_{1,k}^{n_j}(s_{n_j,m_j}) g_{2,k}^{m_j}(s_{n_j,m_j})} 
\ge \int_{-1}^1 e^{- t^2} \, dt
\end{equation*}
for $k=-1,1$.  Combining as in \eqref{E:keycomb} we find that 
\begin{equation*}
\liminf \|\Lop_{n_j,m_j}\|_\mu^2\ge \frac{\sqrt{\breve p\left(\frac{u}{1+u}\right) \breve p^*\left(\frac{u}{1+u}\right)}}{5}.
\end{equation*}
Since the right-hand side above approaches infinity as $u\to1$ we see 
from Theorem  \ref{T:Lnrm} that $\Lop$ fails to be bounded on $L^2(\bndry D,\mu)$.
\end{Ex}

\begin{Ex}\label{E:x2}
Pick $0<\nu<1$ and let $\breve p$ be a continuous map from $[0,1]$ to $[2,\infty]$ satisfying 
\begin{itemize}
\item $\breve p(s)=\left(s-\frac12\right)^{-\nu}$ for $s$ near $\frac12$;
\item $\breve p(s)$ is finite and smooth for $s\ne \frac12$;
\item $\breve p(s) \equiv 2$ for $s$ near $0$ and for $s$ near $1$.  
\end{itemize}
Let $D$ be the domain generated by $\breve p, 1, 1$ as in 
\eqref{E:s-theta-param} and Definition \ref{D:gen}. 
We claim that $D$ has the following properties:
\refstepcounter{equation}\label{N:exprop2}
\begin{enum}
\item $D$ is a convex Reinhardt domain; \label{I:strong}
\item $\bndry D$ is of class $C^{1,\frac{1}{\nu+1}}$ (but not better);\label{I:nst}
\item\label{I:sat} $\bndry D$ is of class $C^{\infty}$ away from the torus corresponding to $s=\frac12$; 
\item $D\notin \tilde \RR$; \label{I:ntr2}
\item the principal curvatures $\kappa_1$ and $\kappa_2$ have positive lower and upper bounds, while $\kappa_3$ has a positive lower bound but tends to infinity we approach the  torus 
corresponding to $s=\frac12$;
\item the conditions of Theorem \ref{T:adm} are still equivalent and thus can still be used to define the notion of an admissible measure as in Definition \ref{D:adm}; \label{I:adm2}
\item measures of order $q$  are admissible if and only if $|q|<1/\nu$; \label{I:qadm2}
\item $\Lop_D$ fails to be bounded on $L^2(\bndry D,\mu)$ 
when $\mu$ is any admissible measure  given by $\frac{1}{4\pi^2}\omega(s)\,ds\,d\theta_1\,d\theta_2$ with $\omega(s)$ positive and continuous for $s\in(0,\frac12)\cup(\frac12, 1)$.\label{I:unbdd2}
\end{enum}

Item \itemref{N:exprop2}{I:sat} is clear from the construction. Item\itemref{N:exprop2}{I:nst} may be verified with the use of the series expansion
\begin{equation*}
\phi(r_1)=\phi(r_1^*) - (r_1-r_1^*) - Q(r_1-r_1^*)^{1+\frac{1}{\nu+1}}+\dots
\end{equation*}
where $r_1^*$ is the value of $r_1$ corresponding to $s=\frac12$ and $Q$ is a positive constant.

The other items are verified as in Example \ref{E:x1}. 

We note that $D$ is strongly convex in the sense of \cite{Pol}.
\end{Ex}

\begin{Ex}\label{E:x3} Let 
$$\breve p(s)=\dfrac{\log(10/s)}{\log(10/s)-1/2}, \,b_2=1, \,b_1=\sqrt{\log10}$$
 and let $D$ be the domain generated by $\breve p, b_1, 1$ as in \eqref{E:s-theta-param} and Definition \ref{D:gen}.
We claim that $D$ has the following properties:
\refstepcounter{equation}\label{N:exprop3}
\begin{enum}
\item $D$ is a convex Reinhardt domain;
\item $\bndry D$ is of class $C^1$ but not in any stronger H\"older class; \label{I:hold}
\item $D\in \tilde\RR\setminus \RR$; \label{I:rtnr}
\item the conditions of Theorem \ref{T:adm} are still equivalent and thus can still be used to define the notion of an admissible measure as in Definition \ref{D:adm}; \label{I:adm3}
\item surface measure is not admissible; \label{I:surf}
\item measures of order $q$ are admissible for $|q|<1$; \label{I:qs}
\item $\Lop_D$ fails to be bounded on $L^2(\bndry D,\mu)$ 
when $\mu$ is any admissible measure given by $\frac{1}{4\pi^2}\omega(s)\,ds\,d\theta_1\,d\theta_2$ with $\omega(s)$ positive and continuous for $s\in(0, 1)$. \label{I:unbdd3}
\end{enum}

Note that 
\begin{align*}
\breve p^*(s)&=2\log(10/s),\\
\frac{1}{s\breve p(s)}&=\frac{d}{ds}\left(\log s + \frac12 \log(\log(10/s))\right),\\
\frac{1}{s\breve p^*(s)}&=\frac{d}{ds}\left(- \frac12 \log(\log(10/s))\right).
\end{align*}
 It is easy to check now that  conditions \eqref{E:a1-stretch} through \eqref{E:vert-pb} hold but \eqref{E:p-ext} fails, showing that \itemref{N:exprop3}{I:rtnr} holds.  
 
Away from the $\zeta_2$-axis $D$ behaves like a domain in $\RR$.  

To understand the behavior near the $\zeta_2$-axis we note that \[r_1=s\sqrt{\log(10/s)}\] (by \eqref{E:r1-s}), while $r_2\to 1$ as $s\to 0$.  Item \itemref{N:exprop3}{I:hold} can now be deduced from \eqref{E:hidden-u-var}.
Using \eqref{E:ord-q} we see that a measure of order $q$ takes the form \eqref{E:meas1} with $\omega(s)$ a positive continuous  multiple (near $s=0$) of \[\big(s\log(10/s) \big)^q;\] it follows easily that such a measure is admissible if and only if $|q|<1$, establishing \itemref{N:exprop3}{I:surf} and \itemref{N:exprop3}{I:qs}.

The proof of \itemref{N:exprop3}{I:unbdd3} goes along the same lines as the proof of (6.1g),
 but this time we let $u$ approach $0$.

The other items are verified as in the previous examples.

\end{Ex}

\section{Duality}\label{S:rr}

Given a bounded convex Reinhardt domain $D\subset\C^2$, the {\em polar} of $D$ is the bounded convex Reinhardt domain
\begin{equation} \label{E:polar-def}
D^* \eqdef \{ z\in\C^2 \st \Re \langle z, \zeta \rangle < 1 \text{ for all } \zeta\in D\},
\end{equation}
where $\langle \cdot, \cdot \rangle$ denotes the standard Hermitian inner product on $\C^2$.

For $\zeta\in\bndry D$ there is $z\in\bndry D^*$ satisfying
$\Re \langle z, \zeta \rangle = 1$; the rotational symmetries of $D$ in fact imply that
\begin{equation}\label{E:Tdef}
\langle z, \zeta \rangle = 1.
\end{equation}

Now assume $bD$ is $C^1$-smooth;
then \eqref{E:Tdef} uniquely determines $z\in bD^*$ (the tangent space to $\bndry D$ at $\zeta$ is given by $\Re \langle z, \zeta \rangle = 1$).

Assume further that $D$ is strictly convex (i.e., $\bndry D$ contains no line segments).  Then the map $T:\bndry D\to\bndry D^*$ defined by $T(\zeta)=z$ is injective (since, for $z$ and $\zeta$ as in  \eqref{E:Tdef}, we have $ \{\eta\in \overline D\st \Re \langle z, \eta \rangle = 1\}=\{\zeta\}$).  
Compactness arguments show that $T$ is a homeomorphism.
It is easy to check that $T$ restricts to a homeomorphism from $\gamma=\gamma_D$ to $\gamma^*\eqdef\gamma_{D^*}$ (see
 \eqref{E:gamma-def}); moreover, the restriction of $T$ to $\gamma$ determines the whole map $T$ via the formula 
\begin{equation*}
 T:(r_1 e^{i\theta_1}, r_2 e^{i\theta_2})
\mapsto
\left( T_1(r_1,r_2) e^{i\theta_1}, T_2(r_1,r_2) e^{i\theta_2}\right).
\end{equation*}

As before, we set $\gamma_+=\gamma\cap \R^2_+$ and $\gamma^*_+=\gamma^*\cap \R^2_+$.

\begin{Thm}\label{T:duality} Suppose $D\in\tilde\RR$. Then the following will hold.
\begin{enumerate}[{\normalfont (a)}]
\item  $D^*\in\tilde\RR$.  \label{I:trr*}

\item If $D\in\RR$ then $D^*\in\RR$. \label{I:rr*}

\item The mapping $\displaystyle{T\big|_{\gamma_+}:\gamma_+\to\gamma^*_+}$ is a $C^1$-smooth diffeomorphism. \label{I:Tdiff}

\item The following relations hold along $\gamma_+$: 
\begin{align}
r_1\cdot (r_1\circ T) + r_2\cdot(r_2\circ T) &= 1\label{E:bas*}\\
(r_1\circ T) \,d r_1 + (r_2\circ T)\,dr_2 &= 0\label{E:d*bas*}\\
r_1\,d (r_1\circ T) + r_2\,d(r_2 \circ T) &= 0\label{E:dbas*}\\
r_1\cdot (r_1\circ T)&=s\label{E:r1s}\\
r_2\cdot (r_2\circ T)&=1-s\label{E:r2s}\\
s\circ T&= s\label{E:ss}\\
\frac{1}{p}+\frac{1}{p\circ T} &= 1.\label{E:pp*}
\end{align}

\item If $D$ is generated by $\breve p,b_2, b_1$ then $D^*$ is generated by $\breve p^*,b_2\inv, b_1\inv$. \label{I:gen*}

\item If $d\mu=\omega(s)\,ds\,d\theta_1\,d\theta_2$ describes an admissible measure on $\bndry D$ then $d\tilde\mu=\frac{1}{\omega(s)}\,ds\,d\theta_1\,d\theta_2$ describes an admissible measure on $\bndry D^*$. \label{I:mu*}

\item The operator
\begin{align*}
U_\mu: f \mapsto (f\circ T)\cdot\omega\inv
\end{align*}
defines an isometry between $L^2(\bndry D^*,\tilde\mu)$ and $L^2(\bndry D,\mu)$. \label{I:twine1}

\item If   $\Lop_D$ is bounded on $L^2(\bndry D,\mu)$ 
then
 $\Lop_{D^*}$ is bounded on $L^2(\bndry D^*,\tilde\mu)$ and the isometry $U_\mu$ intertwines $\Lop$ with its adjoints in the following sense: 
\begin{align*}
\Lop^*_{D,\mu}\circ U_\mu &= U_\mu \circ \Lop_{D^*}\\
\Lop_{D}\circ U_\mu &= U_\mu \circ \Lop^*_{D^*,\tilde\mu}.
\end{align*}
The norm of $\Lop_{D^*}$  on  $L^2(\bndry D^*,\tilde\mu)$ equals
 the norm of $\Lop_{D}$ on \linebreak $L^2(\bndry D,\mu)$, and the spectral data for  $\Lop_{\tilde\mu}^*\Lop$  and $\Aop_{\tilde\mu}$ on $\bndry D^*$ match the spectral data for $\Lop_\mu^*\Lop$ and $\Aop_\mu$ on $\bndry D$, respectively. \label{I:twine2}
\end{enumerate}
\end{Thm}

\begin{proof}
The relation \eqref{E:bas*} follows from  \eqref{E:Tdef}.

Holding $z$ fixed in \eqref{E:Tdef} and differentiating with respect to $\zeta\in \gamma_+$ we obtain \eqref{E:d*bas*}.

Solving \eqref{E:Tdef} and \eqref{E:d*bas*} for $r_1\circ T$ and $r_2\circ T$ and recalling \eqref{E:subst2} we obtain
\begin{align*}
r_1\circ T &= \frac{dr_2}{r_1\,dr_2-r_2\,dr_1}=\frac{s}{r_1}\\
r_2\circ T &= \frac{-dr_1}{r_1\,dr_2-r_2\,dr_1}=\frac{1-s}{r_1}\\
\end{align*}
establishing \eqref{E:r1s} and  \eqref{E:r2s}.  It follows that $T$ is a $C^1$-smooth map  from $\gamma_+$ to $\gamma^*_+$; thus we may also hold $\zeta$ fixed in \eqref{E:Tdef} and differentiate with respect to $z$ to obtain \eqref{E:dbas*}.
Define $s$ on $\gamma^*_+$ by using the middle third of
 \eqref{E:subst2};
applying \eqref{E:dbas*}, \eqref{E:r1s} and \eqref{E:r2s} to \eqref{E:subst2} we verify \eqref{E:ss}.

To verify \eqref{E:pp*} we note that from \eqref{E:ds/dt}, \eqref{E:r1s} and \eqref{E:ss} we have 
\begin{align*}
\frac{1}{p}+\frac{1}{p\circ T}
&= \frac{d\log r_1}{d\log s}+ \frac{d\log (r_1\circ T)}{d\log (s\circ T)}\\
&=\frac{d\log r_1}{d\log s}+ \frac{d\log s-d\log r_1}{d\log s}\\
&=1.
\end{align*}

From \eqref{E:ss} and \eqref{E:pp*} we obtain $\breve p_{D^*}=\breve p^*_D$. Parts \eqref{I:trr*} and \eqref{I:rr*} 
of the current theorem follow now from Theorems \ref{T:rr-tilde-cond} and \ref{T:p-rr-cond}; using the limits from the proof of Lemma \ref{L:squ} to sort out the $b_j$s we also obtain \eqref{I:gen*}.

Reversing our reasoning we see that $T\inv$ is also $C^1$-smooth on $\gamma^*_+$.

Item \eqref{I:mu*} is an immediate consequence of Definition \ref{D:adm}, item \itemref{N:adm}{I:adm}
of Theorem \ref{T:adm} and the relation \eqref{E:ss}.

A direct computation shows that the operator $U_\mu$ defined in item \eqref{I:twine1} is norm-preserving.

The intertwining relations in item \eqref{I:twine2} are verified by checking each Fourier piece using \eqref{E:kappa-tau}, \eqref{E:kaptau} and \eqref{E:kaptau*}.  The remaining claims in \eqref{I:twine2} follow from the isometric nature of $U_\mu$ and general principles.
\end{proof}

\begin{Remark}
Aspects of the duality presented here are treated for smooth strongly $\C$-linearly convex domains in arbitrary dimension without the Reinhardt assumption in \cite{Bar2}.
\end{Remark}

\section{Closing remarks}\label{S:conc}

\begin{itemize}

\item[(A)] The following result highlights the special role played by the measure $\mu_0$ given by $d\mu_0=\frac{1}{4\pi^2} \,ds\,d\theta_1\,d\theta_2$.

\begin{Prop}
Let $D\in\tilde\RR$ and suppose that $\Lop$ is bounded on $L^{2}(\bndry D,\mu)$ for some admissible measure $\mu$ on $\bndry D$.  Then the following conditions are equivalent.
\refstepcounter{equation}\label{N:closer}
\begin{enum}
\item\label{I:mult} The restriction of the operator $\Lop_\mu^*\Lop$ to the orthogonal complement of its kernel admits an orthogonal basis of eigenfunctions that consists of $(n,m)$-monomials and is closed under multiplication.
\item\label{I:spec} The measure $\mu$ is a constant multiple of $\mu_0$.
\end{enum}
\end{Prop}

\begin{proof}
Suppose that \itemref{N:closer}{I:spec} holds. Then referring to identity \eqref{E:kappa-tau} and Proposition \ref{P:l*l} 
we see that the functions
\[
\left\{ \left(\dfrac{  s}{r_1}\right)^{n}\left(\dfrac{1- s}{ r_2}\right)^{m}
e^{i(n\theta_1+m\theta_2)}\st
n,m\ge 0\right\}
\]
provide the desired basis.

Suppose now that \itemref{N:closer}{I:mult} holds.  Referring again to 
\eqref{E:kappa-tau} and
Proposition \ref{P:l*l} we see that our basis must contain constant multiples of the eigenfunctions
$\frac{1}{\omega(s)}$ and $\frac{s}{r_1} \frac{1}{\omega(s)} e^{i\theta_1}$, and that furthermore the product of these two eigenfunctions must be a constant multiple of the second eigenfunction.  It follows that $\omega(s)$ is constant, as claimed.
\end{proof}

\item[(B)] The methods we have employed here rely significantly on the circular symmetry of complete Reinhardt domains. We plan  to examine in a future paper the question of which of our results generalize to non-Reinhardt $\C$-linearly convex domains.  Of course, it will also be interesting to see what happens in higher dimension.
\bigskip

\end{itemize}

\end{document}